\documentclass[preprint,3p]{elsarticle}
\usepackage{appendix}
\usepackage{amsmath,amsthm,amsfonts,amssymb,lineno,mathrsfs}
\usepackage{hyperref}
\usepackage{graphicx}
\usepackage{epsfig}
\usepackage{xcolor}
\usepackage{mathtools}
\usepackage[english]{babel}
\usepackage[latin1]{inputenc}

\biboptions{sort&compress}
\journal{J Funct Analysis}
\newtheorem{theorem}{Theorem}[section] 
\newtheorem{corollary}{Corollary}[section] 
\newtheorem{definition}{Definition}[section] 
\newtheorem{lemma}{Lemma}[section] 
\newtheorem{proposition}{Proposition}[section] 
\newtheorem{remark}{Remark}[section] 
\newtheorem{conditions}{Conditions}[section] 
\numberwithin{equation}{section} \numberwithin{theorem}{section}
\numberwithin{lemma}{section} \numberwithin{corollary}{section}
\numberwithin{definition}{section} \numberwithin{remark}{section}
\numberwithin{proposition}{section}
\numberwithin{conditions}{section}

\numberwithin{assumptions}{section}
\begin{document}%
%
%
%
%
%
%
%
%
%

\begin{frontmatter}
\title{Kurzweil's impulsive inverse problem for the Kelvin--Voigt equations with $p$-Laplacian}

\author[Address01]{Stanislav Antontsev}
\ead{antontsevsn@mail.ru}
\author[Address02]{Pardeep Kumar}
\ead{s.pardeep82@gmail.com}

\author[Address03]{Ivan Kuznetsov}
\ead{ivan.usta@math.kz}
\address[Address01]{Voronezh State University, Voronezh, Russian Federation}
\address[Address02]{Independent Researcher, India}
\address[Address03]{Institute of Mathematics and Mathematical Modeling, Almaty, Kazakhstan}
\begin{abstract} 
In the present paper, we proceed with the research initiated in (S. Antontsev \emph{et al.}, Inverse Problems, 2021). We study  the inverse problem for the Kelvin--Voigt equations with $p$-Laplacian and with a right-hand side in the integral over-determination condition, which approximates a discontinuous function at $t=0$. We rescale the problem in a neighbourhood $[0,1/n]$ and pass to the limit as $n$ approaches infinity. This leads to the inverse problem on an infinitesimal initial layer, which transforms the initial velocity into a new one. 
\end{abstract}
\end{frontmatter}

\section{Introduction}\label{sec1}
In the present paper, we study an inverse problem for the Kelvin--Voigt equations with $p$-Laplacian. Let $\Omega \subset \mathbb{R}^d, \ d \in \{2,3\}$ be a bounded domain with a smooth boundary $\partial\Omega\in C^2$. We will use the notations $Q_{T}:=\Omega\times(0,T)$,  $\Gamma_{T}:=\partial\Omega\times(0,T)$ and $T>0$.  We study the following inverse problem for the Kelvin--Voigt equations with $p$-Laplacian: 
\begin{equation}
\label{eq1.1}
\begin{cases}
&\begin{split} &\partial_t\mathbf{v}
-\mu \ \mathrm{div}\big(|\mathbf{D}(\mathbf{v})|^{p-2}\mathbf{D}(\mathbf{v})\big) 
+(\mathbf{v} \cdot \nabla)\mathbf{v}
-\varkappa\Delta \partial_t\mathbf{v}+\nabla\pi
=\varphi(t)\,\mathbf{f}(x,t) \ \
\text{ in }Q_{T},
\end{split}
\\
&\operatorname{div}\mathbf{v}=0 \ \ \text{ in }Q_{T},
\\
&\mathbf{v}(x,0)=\mathbf{v}_0(x) \ \ \text{ in }\Omega,
\\
&\mathbf{v}(x,t)=\boldsymbol{0} \ \ \text{ on }\Gamma_{T},
\end{cases}
\end{equation}
with the integral over-determination condition
\begin{equation}\label{eq1.2}
\int_\Omega \mathbf{v}(x,t)\cdot(\boldsymbol{\omega}(x)-\varkappa\Delta\boldsymbol{\omega}(x))\,dx=j(t) \ \ \text{ in }(0,{T}).
\end{equation}
Here, $\mathbf{v}(x,t)$, $\pi(x,t)$, and $\varphi(t)$ represent the velocity  field, the pressure field, and the time-dependent part of the external forcing function, respectively. The term $\mathbf{D}(\mathbf{v})=\frac{1}{2}(\nabla \mathbf{v}+\nabla \mathbf{v}^{\mathrm{T}})$ is the rate of strain tensor, $\mathbf{f}(x,t)$ and $\boldsymbol{\omega}(x)$ are given vector-valued functions, $j(t)$ is a given scalar-valued function, and $\mathbf{v}_0(x)$ is a given initial velocity field.  The positive constant $\mu$ denotes the dynamic viscosity, while $\varkappa$ is a length scale parameter that characterizes the elasticity of the fluid and the ratio $\frac{\varkappa}{\mu}$ represents the relaxation time, the exponent $p\in(1,2]$. To obtain the uniqueness of the pressure $\pi$, one can impose the condition $\int_{\Omega}\pi({x},t)\,d{x}=0,$ for $t\in [0,T]$. In order to deal with a global weak solution, see \cite[Sec. 8]{Antontsev2023}, we take in \eqref{eq1.1} the function $\mathbf{f}$ in the following form
\begin{equation}
\label{eq1.2'}
\mathbf{f}(x,t)=\boldsymbol{\omega}(x)-\varkappa \Delta \boldsymbol{\omega}(x).
\end{equation}

In this paper, we study the nonlinear impulsive inverse problem aimed at recovering the velocity field $\mathbf{v}(x,t)$, the pressure field $\pi(x,t)$, and the time dependent part $\varphi(t)$ in the external forcing term of the momentum equation, from the system \eqref{eq1.1} with the integral over-determination condition \eqref{eq1.2}, and the given constants $\mu$, $\varkappa$, and functions $\boldsymbol{\omega}(x)$, $\mathbf{v}_0(x),$ and $j(t)$.

Inverse problems involving the determination of a coefficient or the right-hand side of a differential equation arise in the mathematical modeling of physical processes when, in addition to solving the equation, it is necessary to recover certain physical properties of the system or the effects of external sources of the system. The system \eqref{eq1.1} represents a non-Newtonian fluid model describing the motion of incompressible viscoelastic fluids. 

Various direct problems  for classical Kelvin--Voigt equations, i.e., when the right-hand side is given and $p=2$ in \eqref{eq1.1}, have been studied in \cite{Oskolkov1977,Oskolkov1985,Oskolkov1995,Oskolkov1997}.  See also  recent results on a thorough analysis of direct problems \cite{bib-4,bib-5,bib-7,bib-8,Antontsev2021}.

According to the classical theory of inverse problems, the right-hand side of an integral over-determination condition should be continuous  or belong to an appropriate Sobolev space. We recommend the monograph \cite{Prilepko} for a broad overview of this subject. For a continuous right-hand side $j(t)$ in integral over-determination condition \eqref{eq1.2}, inverse problems were studied for the Navier--Stokes equations, i.e., $p=2$ and $\varkappa=0$ in \eqref{eq1.1}, in \cite{Prilepko}, for the Kelvin--Voigt equations, i.e., $p=2$ in \eqref{eq1.1}, in \cite{AntKh,Khom,Kumar-2021}, for the Kelvin--Voigt equations with $p$-Laplacian and damping term in \cite{AntKh}, and for the convective Brinkman--Forchheimer (CBF) equations in \cite{Kumar} and the references therein. The authors established a local in time existence and uniqueness of the inverse problem for Kelvin--Voigt fluids with memory term by the contraction principle in \cite{Kumar-2021}.
In \cite{AntKh}, the authors studied the existence and uniqueness of the inverse problem for Kelvin--Voigt equations with the $p$-Laplace diffusion and damping term. 
Recently, the well-posedness of an inverse problem for two-dimensional CBF equations with integral overdetermination data was established in \cite{Kumar} using an extension of the contraction mapping theorem. 

For a comprehensive overview of inverse problems for the Navier--Stokes system and related models, particularly those concerned with recovering external force densities or equation coefficients using integral or final  over-determination, the reader is referred to \cite{Khom,Chebotarev-2012,Kumar,Prilepko,Vasin,Antontsev2023,Antontsev2024}, and the references therein. See also publications on inverse problems for parabolic and pseudoparabolic equations \cite{AAA,KHSI,KHM,bib-18,bib-19,bib-20,bib-21,bib-9,bib-10}.

In contrast to the results obtained for the 2D Navier--Stokes equations in \cite{Prilepko,Vasin}, for the generalized Kelvin--Voigt equations with $p$-Laplacian and damping term in \cite{AntKh} and for the 2D CBF equations in \cite{Kumar}, where the right-hand side of the integral over-determination is assumed to be smooth, our results for the impulsive inverse problem of the 2D Kelvin--Voigt equations with $p$-Laplacian remain valid even when the right-hand side of the integral over-determination condition is discontinuous at $t=0$.

There are at least two approaches to impulsive differential equations. The first approach is based on classical theory \cite{Lakshmikantham1989}, where the discontinuities of an unknown solution are described by algebraic conditions on its one-sided traces at fixed time points $t=\tau_i$, $i\in I\subseteq\mathbb{N}$. A given discontinuity is called an impulse. 
Mathematical questions of the stability and blow-up of an unknown solution under multiple impulses have been broadly studied. The second approach was proposed by J. Kurzweil in \cite{Kurzweil1958,Kurzweil2011}. Discontinuities of an unknown solutions are provoked in the limit by a singular source term, approximating the Dirac delta function at various fixed time points. Instead of multiple impulses, we treat a single impulse at the initial time. The simple example is 
\begin{equation} 
\label{ex}
{x}^\prime(t)=j'(t)x(t),
\end{equation}
where $j(t)$ approximates the Heaviside function. Instead of $j'(t)$ we can take a sequence $n\mathbf{1}_{(0,1/n)}(t)$, which approximates the Dirac delta function at $t=0$.
Here, we use the well-known result that a weak derivative of the Heaviside function is the Dirac delta function $\delta_{(t=0)}$.
In the limit, this leads to the natural algebraic condition at the discontinuity 
$$
x(0^+)=x(0)e.
$$
In cases, where we cannot split independent variables, for example in 
\[
{x}_n^\prime(t)=f(x_n(t))+n\mathbf{1}_{(0,1/n)}(t) { g(x_n(t))},
\]
we introduce a fast time variable $\overline{t}=nt\in[0,1]$ and a scaled solution $\bar{x}(\overline{t})=x(\overline{t}/n)$:  
\[
\bar{x}_n^\prime(\overline{t})=\frac1n f(\bar{x}_n(\overline{t}))+g(\bar{x}_n(\overline{t})).
\]
In the limit as $n\to\infty$, this leads to the solution of the infinitesimal initial layer:
\[
\bar{x}^\prime(\overline{t})=g(\bar{x}(\overline{t})).
\]
The final trace of a solution at $\overline{t}=0$ is a right-sided trace of the limit discontinuous solution at $t=0$.  

To avoid mixing such theory with the classical theory of impulsive differential equations, where algebraic conditions are artificial and do not follow from passing to the limit in a singular term, we call such equations Kurzweil's impulsive differential equations. We have applied this approach to various types of PDEs. In \cite{AKSS-2024}, we studied the direct initial boundary value problem similar to \eqref{eq1.1}: 
\begin{equation*}
\partial_t\mathbf{v}
-\mu \Delta\mathbf{v}+(\mathbf{v} \cdot \nabla)\mathbf{v}
-\varkappa\Delta \partial_t\mathbf{v}+\nabla\pi
=\varphi(t)\mathbf{v} \ \ 
\text{ in }Q_{T},
\end{equation*}
where $\varphi(t)$ approximates the Dirac delta function $\delta_{(t=0)}$. 

In inverse problems, an unknown term on the right is singular because the right-hand side of the over-determination condition approximates the discontinuity at the initial time. In comparison with direct problems, we cannot assume that the unknown term approximates the Dirac function. Since the term $\varphi$ can be expressed in an explicit form
\begin{align*}
\varphi(t,[\mathbf{v}])=\frac1{
\|\boldsymbol\omega\|_{\mathbb{H}}^2+
\varkappa\|\nabla\boldsymbol\omega\|_{\mathbb{H}}^2}\left(
\mu\int_\Omega|\mathbf{D}(\mathbf{v})|^{p-2}\mathbf{D}(\mathbf{v}):\nabla\boldsymbol\omega\,dx 
-\int_\Omega((\mathbf{v}\cdot\nabla)\boldsymbol\omega)\cdot\mathbf{v}\,dx+j'(t)\right),
\end{align*}
and {$j'(t)$ approximates the Dirac delta function,  this means that $\varphi$ approximates some unknown functional which can be expressed as a linear combination of the Dirac delta function $\delta_{(t=0)}$ and a function from $L^2(0,T)$, see Remark \ref{rem}.}

This paper is organised as follows. Sec. \ref{sec2} begins with introducing the relevant functional spaces, then formulates the main inverse problem, and establishes its relation to the corresponding functional direct problem. In  Sec. \ref{sec3}, the existence and uniqueness theorem \ref{Theorem.2.1} is proven. The proof is partly based on the results \cite{Vasin,Prilepko}. The main novelty is the estimates, which are uniform in $n\in\mathbb{N}$. In Sec. \ref{sec4}, we pass to the limit as $n\to\infty$. Theorem \ref{theorem3.1} is the main result of this paper. Auxiliary procedures, such as scaling and shifting, are applied. We deduce in the limit, as $n\to\infty$, the linear inverse problem in the infinitesimal initial layer, which gives the new initial velocity in the form: $\mathbf{v}_0({x})+(j(0^+)-j(0))\boldsymbol{\omega}({x})$. Therefore, if there is no discontinuity at $t=0$ for the function $j(t)$, i.e., $j(0^+)=j(0)$, there is no discontinuity for the velocity at $t=0$.   

 \section{The main inverse problem}\label{sec2}
We begin this Sec. by introducing the function spaces and  some important results that will be used throughout the paper. We then present the main inverse problem along with the necessary conditions on the input data required to establish our main results.

\subsection{\bf Function spaces}\label{sub2.1}
\vspace{2mm}
Let $\mathbb{C}_{0}^{\infty}(\Omega;\mathbb{R}^d)$ be the space of all infinitely differentiable functions  ($\mathbb{R}^d$-valued) with compact support in $\Omega\subset\mathbb{R}^d$. Let us introduce the following function spaces
widely used in mathematical fluid mechanics; see \cite{Temam}:
\begin{align*}
&\mathcal{V}:=\Big\{\mathbf{v}\in \mathbb{C}_{0}^\infty(\Omega):\, \operatorname{div} \mathbf{v}=0\Big\},
\\
&\mathbb{H}:=\text{ closure of }\mathcal{V} \text{ in the $(L^2(\Omega))^d$--norm},
\\
&\mathbb{V}_{p}:=\text{ closure of }\mathcal{V}\text{ in the $(W_0^{1,p}(\Omega))^d$--norm},\quad p >1,
\\
&\widetilde{\mathbb{L}}^m:=\text{ closure of }\mathcal{V} \text{ in the $(L^m(\Omega))^d$--norm}, \quad m\in (2,\infty].
\end{align*}
As usual, when $p=2$, we denote $\mathbb{V}_{p}$ by $\mathbb{V}$ instead of $\mathbb{V}_{2}$. Then, under some smoothness assumptions on the boundary, we characterize the spaces $\mathbb{H}$, $\mathbb{V}$, $\widetilde{\mathbb{L}}^m$, and $\widetilde{\mathbb{L}}^\infty$ as 
$$\mathbb{H}=\Big\{\mathbf{v} \in (L^2(\Omega))^d:\, \operatorname{div} \mathbf{v}=0, \ \mathbf{v}\cdot\mathbf{n}\big|_{\partial\Omega}=0\Big\},$$ where $\mathbf{n}$ is the unit outward drawn normal to $\partial\Omega$, and $\mathbf{v}\cdot\mathbf{n}\big|_{\partial\Omega}$ should be understood in the sense of trace in $\mathbb{H}^{-1/2}(\partial\Omega)$ (cf. Theorem 1.2, Chapter 1, \cite{Temam}),   with the norm  $\|\mathbf{v}\|_{\mathbb{H}}^2:=\int_{\Omega}|\mathbf{v}({x})|^2 \,d {x},$
$$
\mathbb{V}=\Big\{\mathbf{v}\in (W_0^{1,2}(\Omega))^d:\, \operatorname{div} \mathbf{v}=0 \Big\},
$$  
with the norm $\|\mathbf{v}\|_{\mathbb{V}}^2:=\int_{\Omega}|\nabla \mathbf{v}({x})|^2 \,d {x},$ 
$$
\widetilde{\mathbb{L}}^m= \Big\{\mathbf{v}\in (L^m(\Omega))^d:\, \operatorname{div} \mathbf{v}=0, \  \mathbf{v} \cdot \mathbf{n} \big|_{\partial\Omega}=0 \Big\},
$$ 
with the norm $\|\mathbf{v}\|_{\widetilde{\mathbb{L}}^m}^m:=\int_{\Omega}|\mathbf{v}({x})|^m \,d {x}$, and 
$$
\widetilde{\mathbb{L}}^\infty=\Big\{\mathbf{v}\in (L^\infty(\Omega))^d:\, \operatorname{div} \mathbf{v}=0, \  \mathbf{v}\cdot \mathbf{n} \big|_{\partial\Omega}=0 \Big\},
$$ 
where $\|\mathbf{v}\|_{\widetilde{\mathbb{L}}^\infty}:=\operatorname*{esssup}\limits_{{x} \in \Omega}|\mathbf{v}({x})|$, respectively. 

Let $(\cdot,\cdot)$ denote the inner product in the Hilbert space $\mathbb{H}$ and $\langle\cdot,\cdot\rangle$ represent the duality pairing between the spaces $\mathbb{V}$ and its dual $\mathbb{V}^{\prime}$. Integration by parts gives 
\begin{align*}
\big((\mathbf{u}\cdot\nabla)\mathbf{v},\mathbf{w}\big) =  -\big((\mathbf{u}\cdot\nabla)\mathbf{w},\mathbf{v}\big), \ \text{ for all } \mathbf{u},\mathbf{v},\mathbf{w}\in \mathbb{V},
\end{align*}
and 
\begin{align*}
\big((\mathbf{u}\cdot\nabla)\mathbf{v},\mathbf{v}\big) = 0, \ \text{ for all } \mathbf{u},\mathbf{v} \in\mathbb{V}.
\end{align*}
We also observe that
\begin{align}\label{Convective}
(\mathbf{v} \cdot \nabla)\mathbf{v}=  \operatorname{div}(\mathbf{v} \otimes\mathbf{v}), \ \text{ for all } \mathbf{v} \in\mathbb{V}.
\end{align}
The following well-known Ladyzhenskaya's inequality is given in \cite[Lemma 1]{Ladyzhenskaya} and \cite[Lemma 5.2, p. 106]{Seregin}:
\begin{equation}
\label{Leq}
\|\mathbf{v}\|_{\widetilde{\mathbb{L}}^4}\leq C\|\mathbf{v}\|_{\mathbb{H}}^{1-\frac{d}4}\|\nabla\mathbf{v}\|_{\mathbb{H}}^{\frac{d}4}, \
\text{ for all }\mathbf{v} \in\mathbb{V},\,\,d=2,3,
\end{equation}
and Poincar\'e's inequality
\begin{align*}
\|\mathbf{v}\|_{\mathbb{H}}\leq\frac{1}{\sqrt{\lambda_1}}\|\nabla\mathbf{v}\|_{\mathbb{H}},  \ \ \text{ for all }  \mathbf{v} \in\mathbb{V},
\end{align*}
where $\lambda_1$ is the smallest eigenvalue of the Stokes operator, is used repeatedly in this study. 

We need two compactness results. Here, the double hook-arrow $\hookrightarrow \hookrightarrow$' denotes that the embedding is compact, and the one hook-arrow $\hookrightarrow$' denotes that the embedding is continuous.

\begin{lemma}[The Aubin--Lions Lemma. $C^0$ version] 
\label{ALSlemma-1}
See \cite[Theorem A.52]{Bedrossian}. Let $X\hookrightarrow\hookrightarrow Y\hookrightarrow Z$ be separable, reflexive Banach spaces, such that the embedding $X\hookrightarrow\hookrightarrow Y$ is compact, and the embedding $Y\hookrightarrow Z$ is continuous. Let $T>0$, and assume that we have a sequence of functions $\{u_n\}_{n \ge1}$ such that there is an $M>0$,
\begin{itemize}
\item $\|u_n(t)\|_{X}\le M$,
\item $\{u_n\}$ is uniformly equicontinuous on $[0,T]$ with values in $Z$.
\end{itemize}
Then, the sequence $\{u_n\}$ is pre-compact in $C([0,T];Y)$.
\end{lemma}

\begin{lemma}[The Aubin--Lions Lemma. $L^p$ version]
\label{ALSlemma}
See \cite[Theorem A.53]{Bedrossian}.  Let $X\hookrightarrow\hookrightarrow Y\hookrightarrow Z$ be separable reflexive Banach spaces. Let $T>0$, and assume that we have a sequence of functions $\{u_n\}_{n\geq1}$ such that
\begin{itemize}
\item $\{u_n\}$ is uniformly bounded in $L^p(0,T;X)$,
\item $\{\partial_tu_n\}$ is uniformly bounded in $L^1(0,T;Z)$,
\end{itemize}
where $1<p<\infty$. Then, the sequence $\{u_n\}$ is (strongly) pre-compact in $L^p(0,T;Y)$.
\end{lemma}

\subsection{\bf Inverse problem}
\vspace{2mm}
In the present paper, we study the family of inverse problems depending on $n \in \mathbb{N}$
\begin{equation}
\label{eq2.1}
\begin{cases}
&\begin{split} &\partial_t\mathbf{v}_{n}-\mu \,\mathrm{div}\big(|\mathbf{D}(\mathbf{v}_n)|^{p-2}\mathbf{D}(\mathbf{v}_n)\big)  +
\mathrm{div}(\mathbf{v}_{n}\otimes\mathbf{v}_{n})
-\varkappa\Delta \partial_t \mathbf{v}_n+\nabla\pi_n
\\
& \qquad=\varphi_n(t)\,\left(\boldsymbol{\omega}(x)-\varkappa\Delta\boldsymbol{\omega}(x)\right) \ \
\text{ in }Q_{T_n},
\end{split}
\\
&\operatorname{div}\mathbf{v}_n=0 \ \ \text{ in }Q_{T_n},
\\
&\int_\Omega \mathbf{v}_n(x,t)\cdot(\boldsymbol{\omega}(x)-\varkappa\Delta\boldsymbol{\omega}(x))\,dx=j_n(t) \ \ \text{ in }(0,{T_n}),
\\
&\mathbf{v}_n(x,0)=\mathbf{v}_0(x) \ \ \text{ in }\Omega,
\\
&\mathbf{v}_n(x,t)=\boldsymbol{0} \ \ \text{ on }\Gamma_{T_n},
\end{cases}
\end{equation}
where $\mathbf{v}_n$ is the unknown velocity field, $\pi_n$ is the unknown pressure field, $\varphi_n(t)$ is an unknown function, and $j_n(t)$ is an approximation of a given function $j(t)$ with a first order discontinuity at $t=0$, $T_n:={T+\frac1n}$ is a given positive constant, $Q_{T_n}:=\Omega\times(0,{T_n})$, $\Omega$ is a bounded domain of $\mathbb{R}^d$,  and $\Gamma_{T_n}:=\partial\Omega\times(0,T_n)$, with $\partial\Omega$ denotes the boundary of $\Omega$.

The nonlinear inverse problem consists of finding a set of the functions $(\mathbf{v}_n,\varphi_n)$, which satisfies the system \eqref{eq2.1}.
\begin{definition}
\label{def1.1}
A pair of functions $(\mathbf{v}_n,\varphi_n)$ is said to be a weak solution of the nonlinear inverse problem \eqref{eq2.1} if
\begin{align*}
\mathbf{v}_n\in L^\infty(0,T;\mathbb{V} \cap \mathbb{V}_p)\cap L^p(0,T;\mathbb{V}_p),
\quad\partial_t\mathbf{v}_n
\in L^2(0,T;\mathbb{V}),
\,\, \quad \varphi_n\in L^2(0,T),
\end{align*}
and they satisfy the integral identity for any $\boldsymbol{\phi}\in\mathbb{V}$
\begin{align}\label{eq2.02}
& \frac{d}{dt}\int_\Omega \Big(\mathbf{v}_n(x,t)\cdot\boldsymbol{\phi}(x)
+\varkappa\nabla\mathbf{v}_n(x,t):\nabla\boldsymbol{\phi}(x)\Big)\,dx
+\mu\int_\Omega |\mathbf{D}(\mathbf{v}_n)(x,t)|^{p-2}\mathbf{D}(\mathbf{v}_n)(x,t):\nabla\boldsymbol{\phi}(x)\,dx
\nonumber\\
&
\quad
-\int_\Omega \mathbf{v}_n(x,t)\otimes\mathbf{v}_n(x,t):\nabla\boldsymbol{\phi}(x)\,dx=\varphi_n(t)\int_\Omega
(\boldsymbol{\omega}(x)-\varkappa\Delta\boldsymbol{\omega}(x))\cdot\boldsymbol{\phi}(x)\,dx
\end{align}
along with the initial condition 
\begin{align*}
&\mathbf{v}_n(x,0)=\mathbf{v}_0(x),
\end{align*}
and the over-determination condition
\begin{align*}
&\int_\Omega 
\mathbf{v}_n(x,t)\cdot\left(\boldsymbol{\omega}(x)-\varkappa\Delta\boldsymbol{\omega}(x)\right)\,dx=j_n(t),
\end{align*}
with a corresponding compatibility condition
\begin{align*}
&\int_\Omega\mathbf{v}_0(x)\cdot\left(\boldsymbol{\omega}(x)-\varkappa\Delta\boldsymbol{\omega}(x)\right)\,dx=j_n(0).
\end{align*}
\end{definition}
\begin{remark} The identity \eqref{eq2.02} can be rewritten in integral form as follows:
\begin{align}\label{eq2.04}
&\int_{Q_{T_n}}  \Big(\partial_t \mathbf{v}_n \cdot \boldsymbol{\eta} +\varkappa \nabla\partial_t \mathbf{v}_n : \nabla \boldsymbol{\eta} +  \mu \, |\mathbf{D}(\mathbf{v}_n)|^{p-2}\mathbf{D}(\mathbf{v}_n) : \nabla \boldsymbol{\eta} -(\mathbf{v}_n \otimes \mathbf{v}_n) : \nabla \boldsymbol{\eta} \nonumber \\&\quad-
\varphi_n(t) \left(\boldsymbol{\omega}(x)-\varkappa\Delta\boldsymbol{\omega}(x)\right) \cdot \boldsymbol{\eta} \Big)\,d{x} dt=0,
\end{align}
for a smooth finite test function $\boldsymbol{\eta}$.
\end{remark}

\subsection{\bf Implied Conditions}
\vspace{2mm}
Here, we imply several conditions on the coefficients. Some of such conditions are important for passing the limit as $n\to\infty$. In the following conditions, along with the function $j_n(t)$, we define the functions  $\overline{j}_n({\overline{t}})=j_n(\frac{{\overline{t}}}n)$, ${\overline{t}}\in(0,1)$, and $\widetilde{j}_n(\widetilde{t})=j_n(\frac1n+\widetilde{t})$, $\widetilde{t}\in (0,T)$.
\begin{conditions}
\label{C}
We deal with conditions ${\bf C}_{1}$--${\bf C}_{9}$:
\begin{itemize}
\item[${\bf C}_{1}.$] Exponent $p \in (1,2]$.
\item[${\bf C}_{2}.$] The initial velocity $\mathbf{v}_0$ belongs to $\mathbb{V} \cap \mathbb{V}_p$ for a weak solution, see Definition \ref{def1.1}.
\item[${\bf C}_{3}.$] The assumption on the function $\boldsymbol{\omega}$:
\[
\boldsymbol{\omega}\in\mathbb{V}_p, 
\quad \nabla\boldsymbol{\omega}\in \widetilde{\mathbb{L}}^\infty, \quad\int_{\Omega}\big(|\boldsymbol{\omega}({x})|^2 +\varkappa|\nabla \boldsymbol{\omega}({x})|^2 \big)\,d{x}=1.
\]
\item[${\bf C}_{4}.$] The sequence of smooth functions $j_n(t)$ approximates the discontinuous function $j(t)$ and satisfies the conditions $j_n(0)=j(0)$, $j_n(t)\equiv j(t)$ for all $t\in[\frac1n,T_n]$. Therefore, $\mathrm{supp}(j_n-j)=[0,\frac1n].$ We assume that $j(t)$ is continuous on $(0,T_n]$, and there exists a right-sided limit $j(0^+)\neq j(0)$.
\item[${\bf C}_{5}.$] The inequality $\int_{0}^{T_n}|j_n'(t)|\,dt\le J\in\mathbb R^+$ holds for every $n\in\mathbb{N}$, where $J$ does not depend on $n$.
\item[${\bf C}_{6}.$] The sequence $\{\overline{j}_n\}_{n\geq n_0}$ converges strongly in $C([0,1])$ to $\overline{j}$, $\overline{j}_n(\overline{t})=j_n({\overline{t}}/n)$.
\item[${\bf C}_{7}.$] The sequence $\{\overline{j}_n'\}_{n\geq n_0}$ converges strongly in $L^2(0,1)$ to $\overline{j}'$.
\item[${\bf C}_{8}.$] The inequality $\left(\int_0^T|\widetilde{j}_n'(\widetilde{t})|^2\,d\widetilde{t}\right)^{\frac{1}{2}}\le J_1\in\mathbb R^+$ 
holds for every $n\in\mathbb{N}$, where $J_1$ does not depend on $n$.
\item[${\bf C}_{9}.$] The sequence $\{\widetilde{j}_n'\}_{n\geq n_0}$ converges strongly in  $L^2(0,T)$ to $j'$. 
\end{itemize}
\end{conditions}
\begin{remark}
\label{rem}
The sequence $j_n'(t)$ is bounded in $L^1(0,1/n)$. This means that for every test function $\phi\in C([0,T])$
\[
\int_0^{1/n}\phi(t)j_n'(t) dt \to \alpha\phi(0)\text{ as }n\to\infty,
\]
where $\alpha\in\mathbb{R}\setminus\{0\}$ is a constant.
This means that $j_n'(t)$ tends in weak-$^*$ sense as $n\to\infty$ to the Dirac delta function $\delta_{(t=0)}$ multiplied by $\alpha$. If $\int_0^{1/n}j_n'(t)\,dt=1$ for all $n\in\mathbb{N}$, then $\alpha=1$.  The trivial case is to take  $j_n(t)=nt\mathbf{1}_{[0,1/n]}(t)+\mathbf{1}_{(1/n,\infty)}(t)$ and, correspondingly, $j_n'(t)=n\mathbf{1}_{[0,1/n]}(t)$. See example \eqref{ex}.
\end{remark}
\begin{remark}
The compatibility conditions hold
\begin{align*}
&j_n(0)=\int_\Omega \mathbf{v}_0 ({x})\cdot \left(\boldsymbol{\omega}(x)-\varkappa\Delta\boldsymbol{\omega}(x)\right)\,d{x},  \ \ \text{and} \\&
\overline{j}_n(1)=\widetilde{j}_n(0)=\int_\Omega\overline{\mathbf{v}}_n({x},1)\cdot \left(\boldsymbol{\omega}(x)-\varkappa\Delta\boldsymbol{\omega}(x)\right)\,d{x}=\int_\Omega \widetilde{\mathbf{v}}_n({x},0) \cdot \left(\boldsymbol{\omega}(x)-\varkappa\Delta\boldsymbol{\omega}(x)\right)\,d{x}.
\end{align*}
\end{remark}

\subsection{\bf Functional direct problem} 
\vspace{2mm}
To prove the existence and uniqueness of a solution to problem \eqref{eq2.1}, we will use the method of reducing this problem to a direct problem but with functional coefficients.
This method has been widely used in \cite{Prilepko}. Let us consider the following functional direct problem, which consists of finding a pair of functions $(\mathbf{v}_n,\pi_n)$ satisfying
\begin{equation}
\label{eq2.1'A}
\begin{cases}
&\begin{split} &\partial_t\mathbf{v}_n-\mu \, \mathrm{div}\big(|\mathbf{D}(\mathbf{v}_n)|^{p-2}\mathbf{D}(\mathbf{v}_n)\big) -\varkappa\Delta \partial_t \mathbf{v}_n+
(\mathbf{v}_{n} \cdot \nabla)\mathbf{v}_{n}
+\nabla\pi_n
\\& \qquad=\varphi_n(t,[\mathbf{v}_n])\,\left(\boldsymbol{\omega}(x)
-\varkappa\Delta\boldsymbol{\omega}(x)\right) \ \ \text{ in }Q_{T_n},
\end{split}
\\
&\operatorname{div}\mathbf{v}_n=0 \ \ \text{ in }Q_{T_n},
\\
&\mathbf{v}_n({x},0)=\mathbf{v}_0({x}) \ \ \text{ in }\Omega,
\\
&\mathbf{v}_n({x},t)=\boldsymbol{0} \ \ \text{ on }\Gamma_{T_n},
\\
&\int_\Omega\pi_n({x},t)\,d{x}=0 \ \ \text{ in }(0,T_n),
\end{cases}
\end{equation}
where $\mathbf{v}_0({x})$ is a given function and the function $\varphi_n(t)$  is determined according to the formula
\begin{align}\label{phi_n}
\varphi_n(t,[\mathbf{v}_n])&=\mu\int_\Omega
|\mathbf{D}(\mathbf{v}_n)|^{p-2}\mathbf{D}(\mathbf{v}_n)
:\nabla\boldsymbol\omega\,dx -\int_\Omega
(\mathbf{v}_n\cdot\nabla)\boldsymbol\omega
\cdot\mathbf{v}_n\,dx+j_n'(t).
\end{align}
Let ${\mathbf{v}_n}(x,t)$ be the solution to problem \eqref{eq2.1'A}. We multiply the first equation in \eqref{eq2.1} by the function $\boldsymbol{\omega}({x})$ and integrate by parts over $\Omega$. Then we arrive to 
\begin{align}
\varphi_n(t,[\mathbf{v}_n])&=
\mu\int_\Omega|\mathbf{D}(\mathbf{v}_n)|^{p-2}\mathbf{D}(\mathbf{v}_n):\nabla\boldsymbol\omega\,dx 
-\int_\Omega((\mathbf{v}_n\cdot\nabla)\boldsymbol\omega)\cdot\mathbf{v}_n\,dx\nonumber
\\
&\quad+\frac{d}{dt}\int_{\Omega}\mathbf{v}_n\cdot
\left(\boldsymbol\omega(x)-\varkappa\Delta\boldsymbol{\omega}(x)
\right)\,dx\label{ph i_n}
\end{align}
From equations   \eqref{phi_n} and \eqref{ph i_n}, we get
$$\frac{d}{dt}\int_{\Omega}\mathbf{v}_n(x,t)
\left(\boldsymbol\omega(x)-\varkappa\Delta\boldsymbol{\omega}(x)\right)dx=j_n'(t).
$$ 
Integrating the last one and taking into account  the initial condition
$$\int_{\Omega}\mathbf{v}_n(x,0)\left(\boldsymbol\omega(x)-\varkappa\Delta\boldsymbol{\omega}(x)
\right)dx=j_n(0),
$$ 
we obtain
$$\int_{\Omega}\mathbf{v}_n(x,t)
\left(\boldsymbol\omega(x)-\varkappa\Delta\boldsymbol{\omega}(x)
\right)dx=j_n(t)\, \, \text{ for all}\,\, t\,\in(0,T_n).
$$ 
This  means that a
solution of  the direct problem \eqref{eq2.1'A} satisfies the integral overdetermination condition.

The following lemma establishes a correspondence between the inverse problem \eqref{eq2.1} and the functional direct problem \eqref{eq2.1'A}-\eqref{phi_n}.
\begin{lemma}\label{le:1}
Let $\left(\mathbf{v}_{n},\varphi_{n}\right)$ be an  arbitrary weak solution to the inverse problem \eqref{eq2.1}. Then $\mathbf{v}_{n}$ is a weak solution to the functional direct problem \eqref{eq2.1'A} with  the function $\varphi_n(t,[\mathbf{v}_{n}])$ defined by the formula \eqref{phi_n}. Respectively, vice versa, every weak solution to {the nonlocal problem \eqref{eq2.1'A}-\eqref{phi_n}} is a solution of inverse problem \eqref{eq2.1}, i.e., it satisfies the integral  condition of over-determination, see third line in \eqref{eq2.1}. Moreover, the uniqueness of the solution to the inverse problem \eqref{eq2.1} implies the uniqueness of the solutions of the { problem \eqref{eq2.1'A}-\eqref{phi_n}} and  vice versa.
\end{lemma}
\noindent{\bf Proof.}
A similar lemma is proved in \cite[Lemma 3.1]{AntKh}, and so we omit it here.
\null\nobreak\hfill\ensuremath{\square}

Taking into account the above lemma, we will further consider the problem  \eqref{eq2.1'A}-\eqref{phi_n}.

\section{The existence result of a weak solution of inverse problem (\ref{eq2.1}) for a fixed $n\in\mathbb{N}$}.\label{sec3}

In this section, we prove the existence of weak  solutions of the functional direct problem \eqref{eq2.1'A}-\eqref{phi_n}, which, due to Lemma \ref{le:1}, holds true for the inverse problem \eqref{eq2.1}.
\begin{theorem}[\textbf{Global in time existence}]
\label{Theorem.2.1} Under Conditions \ref{C}, the problem \eqref{eq2.1} and, respectively, \eqref{eq2.1'A}-\eqref{phi_n} have at least one global in time solution in the sense of Definition \ref{def1.1}. The solution satisfies the following energy estimates
\begin{align}
\label{eq:first-int}
&\underset{t\in(0,T_n)}{\rm esssup}\Big(\|\mathbf{v}_n(\cdot,t)\|_{\mathbb{H}}^2+
\|\nabla \mathbf{v}_n(\cdot,
t)\|_{\mathbb{H}}^2\Big)+\|\nabla\mathbf{v}_n\|_{L^p(0,T_n;\widetilde{\mathbb{L}}^p)}^p
+\|\varphi_n\|_{L^1(0,T_n)}
\leq C,
\\
&\|\partial_t\mathbf{v}_n\|_{L^2(0,T_n;\mathbb{H})}^2+ \|\nabla\partial_t\mathbf{v}_n\|^2_{L^2(0,T_n;\mathbb{H})}+\underset{t\in(0,T_n)}{\rm esssup}\|\nabla\mathbf{v}_n(\cdot,t)\|_{\widetilde{\mathbb{L}}^p}^p
+\|\varphi_n\|_{L^2(0,T_n)}^2
\nonumber\\& \qquad\leq  \frac{\mu}{p}\|\mathbf{v}_0\|_{\mathbb{V}_{p}}^{p}+ C\left(1+\int_{0}^{T_n} |j_n'(t)|^2\,dt\right), \label{eq:second-int}
\\
&\|\partial_t\mathbf{v}_n\|_{L^1(0,T_n;\mathbb{H})}+ \|\nabla\partial_t\mathbf{v}_n\|_{L^1(0,T_n;\mathbb{H})} \leq C, \label{eq:third-int}
\end{align}
where the constant $C$ does not depend on $n$.
\end{theorem}
\begin{remark}
\label{rem3.1}
In comparison with \eqref{eq:second-int}, the estimate
\begin{align*}
&
\|\partial_t\mathbf{v}_n\|_{L^2(1/n,T_n;\mathbb{H})}^2+ \|\nabla\partial_t\mathbf{v}_n
\|^2_{L^2(1/n,T_n;\mathbb{H})}+\underset{t\in(1/n,T_n)}{\rm esssup}\|\nabla\mathbf{v}_n(\cdot,t)\|_{\widetilde{\mathbb{L}}^p}^p 
+\|\varphi_n\|_{L^2(1/n,T_n)}^2
\\
&\qquad\leq  \frac{\mu}{p}\|\mathbf{v}_0\|_{\mathbb{V}_{p}}^{p}+C\left(1+\int_{1/n}^{T_n} |j_n'(t)|^2\,dt\right)
\end{align*}
is uniform in $n$.
\end{remark}

\noindent\textbf{Proof.} The existence of a global solution to problem \ref{eq2.1} for a fixed parameter $n$
was proven in \cite{AntKh}. The global solvability of a weak solution is also guaranteed by condition \eqref{eq1.2'}. This study focuses on proving global estimates of solutions that are uniform in parameter $n$.
\vspace{2mm}

\noindent \textbf{The first uniform energy estimate.} Taking the inner product with $\mathbf{v}_n(\cdot)$ to the first equation in \eqref{eq2.1'A}, and using integration by parts over $\Omega$, we obtain
\begin{align}\label{E1}
\frac{d}{dt}\int_\Omega\Big(|\mathbf{v}_n|^2
+\varkappa|\nabla\mathbf{v}_n|^2\Big)\,dx +
2\mu\int_\Omega|\mathbf{D}(\mathbf{v}_n)|^p\,dx=2\varphi_n(t,[\mathbf{v}_n])j_n(t).
\end{align}
With reference to \eqref{phi_n}, we obtain the following inequality:
\begin{align}
&2\varphi_n(t,[\mathbf{v}_n])j_n(t)=
2j_n(t)\Bigg(\mu\int_\Omega |\mathbf{D}(\mathbf{v}_n)|^{p-2}\mathbf{D}(\mathbf{v}_n) :\nabla\boldsymbol\omega\,d{x}
-\int_\Omega
(\mathbf{v}_n \cdot\nabla)\boldsymbol\omega\cdot\mathbf{v}_n\,d{x} +j_n'(t)\Bigg) \nonumber
\\
& \leq 
2|j_n(t)| \Big( \mu \| \mathbf{D}(\mathbf{v}_n) \|_{\widetilde{\mathbb{L}}^p}^{p-1} \| \nabla \boldsymbol{\omega}\|_{\widetilde{\mathbb{L}}^p} +\| \mathbf{v}_n \|_{\mathbb{H}}^{2} \| \nabla\boldsymbol{\omega} \|_{\widetilde{\mathbb{L}}^\infty}+|j_n'(t)|\Big) \nonumber
\\
&\leq2 \, J_0\Bigg(\frac{\mu\delta}{p^\prime}\|\mathbf{D}(\mathbf{v}_n) \|_{\widetilde{\mathbb{L}}^p}^{p}
+\frac{\mu}{p\delta^{p-1}}\|\nabla\boldsymbol\omega\|_{\widetilde{\mathbb{L}}^p}^p 
+\|\nabla\boldsymbol{\omega}\|_{\widetilde{\mathbb{L}}^\infty}\Big(\|\mathbf{v}_n\|_{\mathbb{H}}^{2}
+\varkappa\|\nabla\mathbf{v}_n\|_{\mathbb{H}}^2 \Big) 
+|j_n'(t)|\Bigg),
\label{E2}
\end{align}
where $\max\limits_{t\in[0,T_n]}|j_n(t)|=J_0$ and $\frac1p+\frac1{p'}=1$. Substituting the estimate \eqref{E2} into equation \eqref{E1} yields the following estimate:
\begin{align}
&\frac{d}{dt} \Big(\| \mathbf{v}_n \|_{\mathbb{H}}^2+\varkappa\| \nabla\mathbf{v}_n \|_{\mathbb{H}}^2 \Big)+ 2\mu\left(1-\frac{J_0\delta}{p^\prime}\right)
\|\mathbf{D}(\mathbf{v}_n)\|_{\widetilde{\mathbb{L}}^p}^p
\nonumber
\\&
\leq 
\frac{2\mu J_0}{p\delta^{p-1}}\|\nabla\boldsymbol\omega\|_{\widetilde{\mathbb{L}}^p}^p
+ 2J_0 \| \nabla\boldsymbol{\omega} \|_{\widetilde{\mathbb{L}}^\infty}  \Big(\| \mathbf{v}_n \|_{\mathbb{H}}^{2}   + \varkappa\|\nabla \mathbf{v}_n\|_{\mathbb{H}}^2 \Big) 
+2J_0|j_n'(t)|.\label{E3}
\end{align}
By assumption $\int_{0}^{T_n} | j_n'(s)|\,ds \le J $, and $\delta_0<1-\frac{J_0\delta}{p^\prime} $, and applying Gronwall's inequality to inequality \eqref{E3}, we get the uniform estimates in 
$$
L^\infty(0,T_n;\mathbb{H})\cap L^\infty(0,T_n;\mathbb{V}) \cap L^{p}(0,T_n;\mathbb{V}_p).
$$ 
It follows from \eqref{phi_n} that  
\begin{align*}
\int_0^{T_n}|\varphi_n(t)|\,dt 
& \leq \int_0^{T_n}\Big( \mu \| \mathbf{D}(\mathbf{v}_n) \|_{\widetilde{\mathbb{L}}^p}^{p-1} \| \nabla \boldsymbol{\omega}\|_{\widetilde{\mathbb{L}}^p} +\| \mathbf{v}_n \|_{\widetilde{\mathbb{L}}^{4}}^{2} \| \nabla\boldsymbol{\omega} \|_{\mathbb{H}}+|j_n'(t)|\Big) \,dt 
\\
&\leq \mu T_n^{\frac{1}{p}}\|\nabla \boldsymbol{\omega}\|_{\widetilde{\mathbb{L}}^{p}} \|\mathbf{v}_n\|_{L^{p}(0,T_n;\mathbb{V}_p)}^{p-1}{ +C T_n
\|\boldsymbol{\omega}\|_{\mathbb{V}}
\|\mathbf{v}_n\|_{L^\infty(0,T_n;\mathbb{H})}^{2-\frac{d}{2}}
\|\mathbf{v}_n\|_{L^\infty(0,T_n;\mathbb{V})}^\frac{d}{2} }
\nonumber 
\\
&
\quad
+\int_0^{T_n}|j'_n(t)|\,dt\leq K_1,
\end{align*}
where we have used H\"older's and  Ladyzhenskaya's  inequalities. This estimate is uniform in $n$, but the estimate 
$\int_0^{T_n}|\varphi_n(t)|^{2}\,dt$ is non-uniform in $n$.
Although, the estimate 
$\int_{1/n}^{T_n}|\varphi_n(t)|^{2}\,dt$ is uniform in $n$.
\vspace{2mm}

\noindent\textbf{The second uniform energy estimate.}
Taking the inner product with $\partial_{t}\mathbf{v}_n(\cdot)$ to the first equation in \eqref{eq2.1'A}, we deduce
\begin{align}\label{E4}
&
\int_\Omega \Big(|\partial_t\mathbf{v}_n|^2
+\varkappa|\nabla\partial_t\mathbf{v}_n|^2 \Big)\,dx+\frac{\mu}{p}\frac{d}{dt}\int_\Omega|\mathbf{D}(\mathbf{v}_n) |^{p} \,dx
\nonumber\\&
=-\int_\Omega(\mathbf{v}_n\cdot
\nabla)\mathbf{v}_n \cdot\partial_t\mathbf{v}_n\,dx +\varphi_n(t,[\mathbf{v}_n]) j_n^{\prime}(t).
\end{align}
Using H\"older's, Ladyzhenskaya's, and Young's inequalities, we obtain the following estimate:
\begin{align}\label{E5}
& \Big| \int_\Omega(\mathbf{v}_n\cdot
\nabla)\mathbf{v}_n \cdot\partial_t\mathbf{v}_n\,dx \Big |= \Big| \int_\Omega(\mathbf{v}_n\cdot
\nabla)\partial_t\mathbf{v}_n \cdot\mathbf{v}_n\,dx \Big | \leq \|\nabla\partial_t\mathbf{v}_n \|_{\mathbb{H}} \|\mathbf{v}_n \|_{\widetilde{\mathbb{L}}^4}^2
\nonumber
\\
& \quad \leq C\|\nabla\partial_t\mathbf{v}_n \|_{\mathbb{H}} \|\mathbf{v}_n \|_{\mathbb{H}}^{2(1-\frac{d}4)} \|\nabla\mathbf{v}_n \|_{\mathbb{H}}^{\frac{d}2} \leq \frac{\varkappa}{2}\|\nabla\partial_t\mathbf{v}_n \|_{\mathbb{H}}^2+ C\|\mathbf{v}_n \|_{\mathbb{H}}^{4-d} \|\nabla\mathbf{v}_n \|_{\mathbb{H}}^{d}.
\end{align}
Applying H\"older's inequality 
\[
\int_\Omega|fg|\,dx\leq\left(\int_\Omega|f|^2\,dx\right)^{\frac12}\left(\int_\Omega|g|^2\,dx\right)^{\frac12},\quad \int_\Omega|fg|\,dx\le \left(\int_\Omega|f|^{\frac{p}{p-1}}\,dx\right)^{\frac{p-1}p}\left(\int_\Omega |g|^p\,dx\right)^{\frac1p},
\]
and Young's inequality 
\[
ab\leq\frac{2(p-1)}{p}a^{\frac{p}{2(p-1)}}+\frac{2-p}{p}b^{\frac{p}{2-p}},
\]
and the trivial inequality
\[
(a+b+c)^2 \leq 3(a^2+b^2+c^2),
\] 
which is a specific application of the Cauchy--Schwarz inequality, and then using  \eqref{phi_n} in the resulting estimate, and Ladyzhenskaya's inequality \eqref{Leq}, we obtain
\begin{align}\label{E6}
&|\varphi_n(t,[\mathbf{v}_n])j_n'(t)|=|\varphi_n(t,[\mathbf{v}_n])||j_n'(t)|=\Big(\sqrt{|\varphi_n(t,[\mathbf{v}_n])|}\sqrt{|j_n'(t)|}\Big)^2\leq \Big(\frac{1}{2}|\varphi_n(t,[\mathbf{v}_n])| + \frac{1}{2}|j_n'(t)| \Big)^2
\nonumber 
\\& 
\leq \left(\frac{\mu}{2}\int_\Omega |\mathbf{D}(\mathbf{v}_n)|^{p-1}|\nabla\boldsymbol\omega|\,d{x}
+\frac{1}{2}\int_\Omega
|\mathbf{v}_n|^2|\nabla\boldsymbol\omega|\,d{x}+|j_n'(t)|\right)^2 \nonumber
\\
&\leq\frac{3\mu^2}{4}\|\mathbf{D}(\mathbf{v}_n)\|_{\widetilde{\mathbb{L}}^p}^{2(p-1)}
\|\nabla\boldsymbol{\omega}\|_{\widetilde{\mathbb{L}}^p}^2+
\frac34\|\mathbf{v}_n\|_{\widetilde{\mathbb{L}}^4}^{4}
\|\nabla\boldsymbol{\omega}\|_{\mathbb{H}}^2+3|j_n'(t)|^2\nonumber
\\
&\leq
\frac{3\mu^2}{4}\frac{2(p-1)}{p}\|\mathbf{D}(\mathbf{v}_n) \|_{\widetilde{\mathbb{L}}^p}^{p}
+\frac{3\mu^2}{4}\frac{(2-p)}{p}\|\nabla\boldsymbol\omega\|_{\widetilde{\mathbb{L}}^{p}}^{\frac{2p}{2-p}}
+ C\|\mathbf{v}_n \|_{\mathbb{H}}^{4-d} \|\nabla\mathbf{v}_n \|_{\mathbb{H}}^{d}\| \nabla\boldsymbol{\omega} \|_{\mathbb{H}}^2
+3|j_n'(t)|^2.
\end{align}
Substituting the estimates \eqref{E5} and \eqref{E6} in \eqref{E4} and then integrating it with respect to time from $0$ to $t$, we find
\begin{align}
&\int_0^{t} \|\partial_t\mathbf{v}_n(s) \|_{\mathbb{H}}^2 \,ds+\frac{\varkappa}{2} \int_0^{t}\|\nabla\partial_t\mathbf{v}_n(s) \|_{\mathbb{H}}^2 \,ds+\frac{\mu}{p}\|\mathbf{D}(\mathbf{v}_n)(\cdot,t)\|_{\widetilde{\mathbb{L}}^p}^p \nonumber\\& 
 \leq \frac{\mu}{p}\|\mathbf{v}_0\|_{\mathbb{V}_{p}}^{p} 
+\frac{3\mu^2}{4}\frac{2(p-1)}{p} \int_{0}^{t}
\|\mathbf{D}(\mathbf{v}_n)(\cdot,s)\|_{\widetilde{\mathbb{L}}^p}^p \,ds 
+\frac{3\mu^2}{4}\frac{(2-p)}{p}T_n
\|\nabla\boldsymbol\omega\|_{\widetilde{\mathbb{L}}^{p}}^{\frac{2p}{2-p}}
 \nonumber
\\& 
\quad+ C(1+\|\nabla \boldsymbol{\omega}\|_{\mathbb{H}}^2 )\int_{0}^{t}  \|\mathbf{v}_n (\cdot,s)\|_{\mathbb{H}}^{4-d} \|\nabla\mathbf{v}_n (\cdot,s)\|_{\mathbb{H}}^{d} \,ds +3\int_{0}^{t}|j_n'(s)|^2\,ds, \label{E7}
\end{align}
for all $t \in [0,T_n]$ and \eqref{eq:second-int} follows by virtue of Gronwall's inequality and the energy estimate \eqref{eq:first-int}. The proof of \eqref{eq:second-int} is completed.

To obtain \eqref{eq:third-int}, we use an alternative method. Taking the inner product with $\partial_{t}\mathbf{v}_n(\cdot)$ to the first equation in \eqref{eq2.1'A} and then using H\"older's, Ladyzhenskaya's, and Young's inequalities, we obtain the following estimate:
\begin{align}
&\int_\Omega \Big(|\partial_t\mathbf{v}_n|^2
+\varkappa|\nabla\partial_t\mathbf{v}_n|^2 \Big)\,dx=
\int_\Omega(\mathbf{v}_n\otimes\mathbf{v}_n):\nabla\partial_t\mathbf{v}_n\,dx
\nonumber
\\& 
\quad
- \mu \int_\Omega |\mathbf{D}(\mathbf{v}_n) |^{p-2} \mathbf{D}(\mathbf{v}_n):\nabla \partial_t \mathbf{v}_n \,dx  
+\varphi_n(t,[\mathbf{v}_n]) j_n^{\prime}(t)
\nonumber
\\& 
\leq \|\nabla\partial_t\mathbf{v}_n \|_{\mathbb{H}} \|\mathbf{v}_n \|_{\widetilde{\mathbb{L}}^4}^2 + \mu \| \mathbf{D}(\mathbf{v}_n) \|_{\widetilde{\mathbb{L}}^{2(p-1)}}^{p-1}\|\nabla\partial_t\mathbf{v}_n \|_{\mathbb{H}} 
+|\varphi_n(t)||j_n'(t)|
\nonumber
\\& 
\leq \frac{\varkappa}{4} \|\nabla\partial_t\mathbf{v}_n \|_{\mathbb{H}}^2 + C\|\mathbf{v}_n \|_{\mathbb{H}}^{4-d} \|\nabla\mathbf{v}_n \|_{\mathbb{H}}^{d} 
+ \frac{\mu^2}{\varkappa}\| \mathbf{D}(\mathbf{v}_n) \|_{\widetilde{\mathbb{L}}^{2(p-1)}}^{2(p-1)}
\nonumber
\\& 
\quad
+\frac{\varkappa}{4}\|\nabla\partial_t\mathbf{v}_n \|_{\mathbb{H}}^2 
+\frac{1}{2}|\varphi_n(t)|^2 + \frac{1}{2}|j_n'(t)|^2,
\label{eq3.11}
\end{align}
and this gives
\begin{align*}
\|\partial_t\mathbf{v}_n \|_{\mathbb{H}}^2 +\frac{\varkappa}{2} \, \|\nabla\partial_t\mathbf{v}_n \|_{\mathbb{H}}^2 \leq C\|\mathbf{v}_n \|_{\mathbb{H}}^{4-d} \|\nabla\mathbf{v}_n \|_{\mathbb{H}}^{d}
+ \frac{\mu^2}{\varkappa}\| \mathbf{D}(\mathbf{v}_n) \|_{\widetilde{\mathbb{L}}^{2(p-1)}}^{2(p-1)} +\frac{1}{2}|\varphi_n(t)|^2 + \frac{1}{2}|j_n'(t)|^2.
\end{align*}
Taking the square root of the above inequality by the rule 
\[
\frac1{\sqrt{2}}(A+B)\leq\sqrt{A^2+B^2}\leq A+B, 
\]
we arrive at the following inequality:
\begin{align}\label{eq3.12}
\frac{1}{\sqrt{2}}\Big(\|\partial_t\mathbf{v}_n \|_{\mathbb{H}} +\sqrt{\frac{\varkappa}{2}} \, \|\nabla\partial_t\mathbf{v}_n \|_{\mathbb{H}} \Big) & \leq C\|\mathbf{v}_n \|_{\mathbb{H}}^{2-\frac{d}2} \|\nabla\mathbf{v}_n \|_{\mathbb{H}}^{\frac{d}2} + \frac{\mu}{\sqrt{\varkappa}}\| \mathbf{D}(\mathbf{v}_n) \|_{\widetilde{\mathbb{L}}^{2(p-1)}}^{(p-1)}
\nonumber
\\& 
\quad+\sqrt{\frac{1}{2}}|\varphi_n(t)|+ \sqrt{\frac{1}{2}}|j_n'(t)|.
\end{align}
Integrating the above estimate \eqref{eq3.12} with respect to time from $0$ to $T_n$ and then using the estimate \eqref{eq:first-int} and condition $2(p-1)\leq p$ in it, one can immediately obtain the energy estimate \eqref{eq:third-int}.
This completes the proof of Theorem \ref{Theorem.2.1}. 
\begin{corollary}
\label{cor3.11}
From uniform boundedness of $\partial_t\mathbf{v}_n$ in $L^1(0,T_n;\mathbb{V})$ follows the integral equi-continuity
\[
\int_0^{T_n-h}\|\mathbf{v}_n(\cdot,t+h)-\mathbf{v}_n(\cdot,t)\|_{\mathbb{V}}\,dt\leq \omega(h),
\]
where $\omega(h)\to0$ uniformly in $n\in\mathbb{N}$ as $h\to0^+$. 
\end{corollary}
\section*{Uniqueness}
\vspace{2mm}
Let us prove the uniqueness of the weak solutions.
Let $\mathbf{v}_n^1(x,t)$ and $\mathbf{v}_n^2(x,t)$ be two weak solutions to problem \eqref{eq2.1'A}-\eqref{phi_n}. Subtracting the equation for $\mathbf{v}_n^2(x,t)$ from the equation for $\mathbf{v}_n^1(x,t)$, and then multiplying the resulting equation by $\mathbf{v}_n (x,t):= \mathbf{v}_n^1 (x,t)-\mathbf{v}_n^2(x,t)$,  we obtain 

\begin{align}\label{Uniqueness}
&\frac{1}{2}\frac{d}{dt} \Big(\| \mathbf{v}_n \|_{\mathbb{H}}^2+\varkappa\| \nabla\mathbf{v}_n \|_{\mathbb{H}}^2 \Big) + \mu\int_\Omega \Big(
|\mathbf{D}(\mathbf{v}_n^1)|^{p-2}\mathbf{D}(\mathbf{v}_n^1) -|\mathbf{D}(\mathbf{v}_n^2)|^{p-2}\mathbf{D}(\mathbf{v}_n^2) \Big)
:\nabla \mathbf{v}_n\,dx
\nonumber\\&\quad
=-\int_\Omega
(\mathbf{v}_n\cdot\nabla)\mathbf{v}_n^{1}
\cdot\mathbf{v}_n\,dx+(\varphi_n^1(t,[\mathbf{v}_n^{1}])-\varphi_n^2(t,[\mathbf{v}_n^{2}]))
\int_\Omega(\boldsymbol{\omega}({x})-\varkappa\Delta \boldsymbol{\omega}({x})) \cdot\mathbf{v}_n\,dx.
\end{align}
Recalling the choice \eqref{phi_n} and the integral form of over-determination condition in
\eqref{eq2.1}, the term, which includes $\varphi_n^1-\varphi_n^2$, disappears in identity \eqref{Uniqueness}, since 
\[
(\varphi_n^1-\varphi_n^2)\int_\Omega(\boldsymbol{\omega}({x})-\varkappa\Delta \boldsymbol{\omega}({x})) \cdot\mathbf{v}_n\,dx=
(\varphi_n^1-\varphi_n^2)(j_n(t)-j_n(t))=0.
\] 
By the monotonicity property, we have the following inequality
\begin{align*}
\int_\Omega \big(
|\mathbf{D}(\mathbf{v}_n^1)|^{p-2}\mathbf{D}(\mathbf{v}_n^1) -|\mathbf{D}(\mathbf{v}_n^2)|^{p-2}\mathbf{D}(\mathbf{v}_n^2) \big):\nabla \mathbf{v}_n\,dx \geq 0,
\end{align*}
see \cite[Lemma 2.3]{AntKh}, and references therein. Using H\"older's, Ladyzhenskaya's, and Poincar\'e's inequalities, we get 
\begin{align*}
&
\Big|\int_\Omega
(\mathbf{v}_n\cdot\nabla)\mathbf{v}_n^{1}
\cdot\mathbf{v}_n\,dx \Big| \leq \|\nabla \mathbf{v}_n^1 \|_{\mathbb{H}} \|\mathbf{v}_n \|_{\widetilde{\mathbb{L}}^4}^2 \leq C \|\nabla \mathbf{v}_n^1 \|_{\mathbb{H}} \| \mathbf{v}_n \|_{\mathbb{H}}^{2-\frac{d}{2}} \|\nabla \mathbf{v}_n \|_{\mathbb{H}}^\frac{d}{2}
\\& \quad \leq C \|\nabla \mathbf{v}_n^1 \|_{\mathbb{H}}  \| \nabla \mathbf{v}_n \|_{\mathbb{H}}^2 \leq C \|\nabla \mathbf{v}_n^1 \|_{\mathbb{H}}  \Big(\| \mathbf{v}_n \|_{\mathbb{H}}^2+\varkappa\| \nabla\mathbf{v}_n \|_{\mathbb{H}}^2 \Big).
\end{align*}
Substituting the above estimates into \eqref{Uniqueness}, we obtain the following estimate:
\begin{align}\label{Uniqueness1}
\frac{d}{dt} \Big(\| \mathbf{v}_n (\cdot,t)\|_{\mathbb{H}}^2+\varkappa\| \nabla\mathbf{v}_n (\cdot,t)\|_{\mathbb{H}}^2 \Big)
\leq C \|\nabla \mathbf{v}_n^1(\cdot,t) \|_{\mathbb{H}}  \Big(\| \mathbf{v}_n (\cdot,t)\|_{\mathbb{H}}^2+\varkappa\| \nabla\mathbf{v}_n (\cdot,t)\|_{\mathbb{H}}^2 \Big),
\end{align}
for all $t \in [0,T_n]$. Applying Gronwall's inequality in \eqref{Uniqueness1}, we obtain
\begin{align}\label{Uniqueness2}
\|\mathbf{v}_n(\cdot,t)\|_{\mathbb{H}}^2 +\varkappa \|\nabla\mathbf{v}_n(\cdot,t)\|_{\mathbb{H}}^2\leq \Big(\|\mathbf{v}_n(\cdot,0)\|_{\mathbb{H}}^2 + \varkappa \|\nabla \mathbf{v}_n(\cdot,0)\|_{\mathbb{H}}^2 \Big) \exp\bigg( C\int_0^{T_n}  \|\nabla \mathbf{v}_n^1(\cdot,t)\|_{\mathbb{H}} \,d t \bigg).
\end{align}
Taking $\mathbf{v}_n^1({x},0)=\mathbf{v}_n^2({x},0)$ in \eqref{Uniqueness2}, we obtain $\mathbf{v}_n^1({x},t)=\mathbf{v}_n^2({x},t)$ and hence, the uniqueness of the weak solutions follows.

\section{The main result}\label{sec4}
\vspace{2mm}
In the present paper, the sequence $\{j_n\}_{n\geq n_0}$ approximates a discontinuous function $j(t)$ at $t=0$ on the interval $(0,\frac1n)$: 
\[
\mathrm{supp}(j_n-j)=\Big[0,\dfrac1n \Big].
\]
Here, we represent the domain $Q_{T_n}$ as the union of two subdomains $\Omega\times[0,\frac1n]$ and $\Omega\times[\frac1n,T_n]$ with intersection at $t=\frac1n$, and formulate the scaled and shifted problems, respectively. Such an approach was elaborated in \cite{AKS} for the impulsive Navier--Stokes system {and in \cite{AKS-2024} for the impulsive Kelvin--Voigt equations.} In the domain $\Omega\times[0,\frac1n]$, for $\overline{t}=tn\in[0,1]$ we introduce the scaled functions
\[
\overline{\mathbf{v}}_n({x},{\overline{t}})=\mathbf{v}_n\Big({x},\frac{{\overline{t}}}n\Big), \ \ 
\overline{\varphi}_n({\overline{t}})=\frac1n\varphi_n\Big(\frac{{\overline{t}}}n\Big), \ \
\overline{j}_n({\overline{t}})=j_n\Big(\frac{{\overline{t}}}n\Big), \ \
\overline{\pi}_n({x},{\overline{t}})=\frac1n\pi_n\Big({x},\frac{{\overline{t}}}n\Big),
\]
satisfying the system
\begin{equation}
\label{eq3.1}
\begin{cases}
&\begin{split} &\partial_{\overline{t}}\overline{\mathbf{v}}_{n}-\frac{\mu}{n}\mathrm{div} \big(|\mathbf{D}(\overline{\mathbf{v}}_n)|^{p-2}\mathbf{D}(\overline{\mathbf{v}}_n)\big) +\frac1n
\mathrm{div}(\overline{\mathbf{v}}_{n}\otimes\overline{\mathbf{v}}_n) -\varkappa\Delta \partial_{\overline{t}}\overline{\mathbf{v}}_{n}
+\nabla\overline{\pi}_n\\&\qquad=\overline{\varphi}_n({\overline{t}})\,
(\boldsymbol{\omega}({x}) -\varkappa\Delta \boldsymbol{\omega}({x}))\ \text{ in }Q_1,
\end{split}
\\
&\operatorname{div}\overline{\mathbf{v}}_n=0 \  \ \text{ in }Q_1,
\\
&\int_\Omega
\overline{\mathbf{v}}_n({x},{\overline{t}})\cdot(\boldsymbol{\omega}({x})-\varkappa\Delta \boldsymbol{\omega}({x}))\,d{x}=\overline{j}_n({\overline{t}}) \ \ \text{ in }(0,1),
\\
&\overline{\mathbf{v}}_n({x},0)=\mathbf{v}_0({x}) \ \ \text{ in }\Omega,
\\
&\overline{\mathbf{v}}_n({x},{\overline{t}})=\boldsymbol{0} \ \ \text{ on }\Gamma_1,
\end{cases}
\end{equation}
where $Q_1:=\Omega\times(0,1)$ and $\Gamma_1:=\partial\Omega\times(0,1)$. {This inverse problem is equivalent to the functional direct problem 
\begin{equation}
\label{eq3.1d}
\begin{cases}
&\begin{split} &\partial_{\overline{t}}\overline{\mathbf{v}}_{n}-\frac{\mu}{n}\mathrm{div} \big(|\mathbf{D}(\overline{\mathbf{v}}_n)|^{p-2}\mathbf{D}(\overline{\mathbf{v}}_n)\big) +\frac1n
\mathrm{div}(\overline{\mathbf{v}}_{n}\otimes\overline{\mathbf{v}}_n) -\varkappa\Delta \partial_{\overline{t}}\overline{\mathbf{v}}_{n}
+\nabla\overline{\pi}_n
\\
&\qquad=\overline{\varphi}_n({\overline{t}},[\overline{\mathbf{v}}_{n}])\,
(\boldsymbol{\omega}({x})-\varkappa\Delta \boldsymbol{\omega}({x})) \ \ \text{ in }Q_1,
\end{split}
\\
&\operatorname{div}\overline{\mathbf{v}}_n=0 \ \ \text{ in }Q_1,
\\
&\overline{\mathbf{v}}_n({x},0)=\mathbf{v}_0({x}) \ \ \text{ in }\Omega,
\\
&\overline{\mathbf{v}}_n({x},{\overline{t}})=\boldsymbol{0} \ \ \text{ on }\Gamma_1,
\end{cases}
\end{equation}
where
\begin{align*}
\overline{\varphi}_n({\overline{t}},[\overline{\mathbf{v}}_n])&=\overline{j}_n'({\overline{t}})
+\frac{\mu}{n}\int_\Omega\big(|\mathbf{D}(\overline{\mathbf{v}}_n)|^{p-2}\mathbf{D}(\overline{\mathbf{v}}_n)\big)({x},{\overline{t}}):\nabla\boldsymbol{\omega}({x})\,d{x} 
\\ &\quad
-\frac1n\int_\Omega
(\overline{\mathbf{v}}_n({x},{\overline{t}})\cdot\nabla)\boldsymbol{\omega}({x})\cdot\overline{\mathbf{v}}_n({x},{\overline{t}})\,d{x}.
\end{align*}}
In the rest of the domain {$\Omega\times[\frac1n,T_n]$}, for $\widetilde{t}=t-\frac1n\in[0,T]$ we apply the function
\begin{align*}
&\widetilde{\mathbf{v}}_n({x},\widetilde{t})=\mathbf{v}_n\Big({x}, \frac{1}n+\widetilde{t}\Big), \ 
\widetilde{\varphi}_n(\widetilde{t})=\varphi_n\Big(\frac{1}n+\widetilde{t}\Big), \ 
\widetilde{j}_n(\widetilde{t})=j_n\Big(\frac{1}n+\widetilde{t}\Big), \
\widetilde{\pi}_n({x},\widetilde{t})=\pi_n\Big({x},\frac{1}n+\widetilde{t}\Big),
\end{align*}
satisfying the system
\begin{equation}
\label{eq3.2}
\begin{cases}
&\begin{split} &\partial_{\widetilde{t}}\widetilde{\mathbf{v}}_{n}-\mu \ \mathrm{div}\big(|\mathbf{D}(\widetilde{\mathbf{v}}_n)|^{p-2}\mathbf{D}(\widetilde{\mathbf{v}}_n)\big)+
\mathrm{div}(\widetilde{\mathbf{v}}_{n}\otimes\widetilde{\mathbf{v}}_n) -\varkappa \Delta \partial_{\widetilde{t}} \widetilde{\mathbf{v}}_{n}+\nabla\widetilde{\pi}_n\\&\qquad=\widetilde{\varphi}_n(\widetilde{t})\,
(\boldsymbol{\omega}({x})-\varkappa\Delta \boldsymbol{\omega}({x})) \ \ \text{ in }Q_{T},
\end{split}
\\
&\operatorname{div}\widetilde{\mathbf{v}}_n=0 \ \ \text{ in }Q_{T},
\\
&\int_\Omega\widetilde{\mathbf{v}}_n({x},\widetilde{t})
\cdot(\boldsymbol{\omega}({x})-\varkappa\Delta \boldsymbol{\omega}({x}))\,d{x}
=\widetilde{j}_n(\widetilde{t}) \ \ \text{ in }(0,T),
\\
&\widetilde{\mathbf{v}}_n({x},0)=\overline{\mathbf{v}}_n({x},1) \ \ \text{ in }\Omega,
\\
&\widetilde{\mathbf{v}}_n({x},\widetilde{t})=\boldsymbol{0} \ \
\text{ on }\Gamma_{T},
\end{cases}
\end{equation}
{which is equivalent to the functional direct problem 
\begin{equation}
\label{eq3.2d}
\begin{cases}
&\begin{split} &\partial_{\widetilde{t}}\widetilde{\mathbf{v}}_{n}-\mu \ \mathrm{div}\big(|\mathbf{D}(\widetilde{\mathbf{v}}_n)|^{p-2}\mathbf{D}(\widetilde{\mathbf{v}}_n)\big)+
\mathrm{div}(\widetilde{\mathbf{v}}_{n}\otimes\widetilde{\mathbf{v}}_n) -\varkappa \Delta \partial_{\widetilde{t}} \widetilde{\mathbf{v}}_{n}+\nabla\widetilde{\pi}_n\\&\qquad=\widetilde{\varphi}_n(\widetilde{t},[\widetilde{\mathbf{v}}_n])\,
(\boldsymbol{\omega}({x})-\varkappa\Delta \boldsymbol{\omega}({x})) \ \ \text{ in }Q_{T},
\end{split}
\\
&\operatorname{div}\widetilde{\mathbf{v}}_n=0 \ \ \text{ in }Q_{T},
\\
&\widetilde{\mathbf{v}}_n({x},0)=
\overline{\mathbf{v}}_n({x},1) \ \ \text{ in }\Omega,
\\
&\widetilde{\mathbf{v}}_n({x},\widetilde{t})=\boldsymbol{0} \ \
\text{ on }\Gamma_{T},
\end{cases}
\end{equation}
where
\begin{align}
\widetilde{\varphi}_n(\widetilde{t},[\widetilde{\mathbf{v}}_n])&=\widetilde{j}_n'(\widetilde{t})
+\mu \int_\Omega|\mathbf{D}(\widetilde{\mathbf{v}}_n)|^{p-2}\mathbf{D}(\widetilde{\mathbf{v}}_n)
: \nabla\boldsymbol{\omega}({x})\,d{x} 
-\int_\Omega(\widetilde{\mathbf{v}}_n({x},\widetilde{t})\cdot\nabla)\boldsymbol{\omega}({x})\cdot\widetilde{\mathbf{v}}_n({x},\widetilde{t})\,d{x}.
\label{sphi_n}
\end{align}}
We now present the main results of this study.
\begin{theorem}\label{theorem3.1} 
Let Conditions \ref{C} be satisfied. Then the following assertions hold true.
\begin{itemize}
\item[{\bf 1.}] The family of scaled solutions {$\displaystyle\{(\overline{\mathbf{v}}_n,\overline{\varphi}_n)\}_{n \geq n_0}$} of problem \eqref{eq3.1} is
strongly-strongly relatively compact in the Cartesian product of the spaces 
$C([0,1];\mathbb{H})\times{L^{2}(0,1)}$. Namely, there exist
subsequences of $\displaystyle\{\overline{\mathbf{v}}_n\}_{n\geq n_0}$ and $\displaystyle\{\overline{\varphi}_n\}_{n\geq n_0}$, still labeled by $n$, and a limit vector-function
$\displaystyle\overline{\mathbf{v}}\in C([0,1];\mathbb{H})$ such that
\begin{subequations}
\begin{align}
\label{eq3.3a}
\overline{\varphi}_n \xrightarrow[n\to\infty]{}  \overline{j}^\prime
& \ \mbox{ strongly in }L^{2}(0,1),
\\
\label{eq3.3b} 
\overline{\mathbf{v}}_n
\xrightarrow[n\to\infty]{}  \overline{\mathbf{v}}
& \ \mbox{ strongly in }C([0,1];\mathbb{H}),
\end{align}
\end{subequations}
where the function $\overline{\mathbf{v}}$ resolves the direct problem
\begin{equation}
\label{eq3.4}
\begin{cases}
&\begin{split} & \partial_{\overline{t}} \overline{\mathbf{v}}-
\varkappa\Delta \partial_{\overline{t}} \overline{\mathbf{v}}
+\nabla\overline{\pi}=\overline{j}^\prime({\overline{t}})\,(\boldsymbol{\omega}({x})-\varkappa\Delta\boldsymbol{\omega}({x})) \ \ \text{ in }Q_1,
\end{split}
\\
&\operatorname{div} \overline{\mathbf{v}}=0 \ \ \text{ in }Q_1,
\\
&\overline{\mathbf{v}}({x},0)=\mathbf{v}_0({x}) \ \ \text{ in }\Omega,
\\
&\overline{\mathbf{v}}({x},{\overline{t}})=\boldsymbol{0} \ \ \text{ on }\Gamma_1,
\end{cases}
\end{equation}
where $\overline{j}({\overline{t}})=\lim\limits_{n\to\infty}\overline{j}_n({\overline{t}})$, $\overline{j}(0)=j(0)$, $\overline{j}(1)=j(0^+)$. 

In the case the pressure does not depend on $x$, i.e., $\nabla\overline{\pi}=0$, problem \eqref{eq3.4} is resolved in an explicit way
\[
(\mathrm{I}-\varkappa\Delta) \overline{\mathbf{v}}({x},{\overline{t}})=(\mathrm{I}-\varkappa\Delta)(\mathbf{v}_0({x})+(\overline{j}({\overline{t}})-j(0))\boldsymbol{\omega}({x})),
\]
and by using the fact $(\mathrm{I}-\varkappa\Delta)^{-1} (\mathrm{I}-\varkappa\Delta)=\mathrm{I}$, we have
\[
\overline{\mathbf{v}}({x},{\overline{t}})=\mathbf{v}_0({x})+(\overline{j}({\overline{t}})-j(0))\boldsymbol{\omega}({x}).
\]
We get the final trace
\begin{equation}
\label{eq3.5}
\overline{\mathbf{v}}({x},1)=\mathbf{v}_0({x})+(j(0^+)-j(0))\boldsymbol{\omega}({x}).
\end{equation}

\item[{\bf 2.}] The family of shifted solutions  {$\displaystyle\{(\widetilde{\mathbf{v}}_n,\widetilde{\varphi}_n)\}_{n\geq n_0}$} of problem \eqref{eq3.2} is strongly weakly relatively compact in the Cartesian product of two spaces $C([0,T];\mathbb{H})\times L^{2}(0,T)$. 
Therefore, there exist subsequences of $\displaystyle\{\widetilde{\mathbf{v}}_n\}_{n\geq n_0}$ and $\{\widetilde{\varphi}_n\}_{n\geq n_0}$, still labeled by $n$, and a limit vector-function $\mathbf{v}\in {C([0,T];\mathbb{H})}\cap L^2(0,T;\mathbb{V})$, and a scalar-function $\varphi\in L^{2}(0,T)$, such that
\begin{subequations}
\label{eq25}
\begin{align}
\widetilde{\mathbf{v}}_n \xrightarrow[n\to\infty]{}  \mathbf{v}\quad
&\text{strongly in }L^2(0,T;\mathbb{H}),
\\ \widetilde{\mathbf{v}}_n \xrightarrow[n\to\infty]{}  \mathbf{v}\quad
&\text{strongly in }{C([0,T];\mathbb{H})},
\\
\widetilde{\mathbf{v}}_n \xrightharpoonup[n\to\infty]{*} \mathbf{v}\quad
& \text{weakly$^*$ in } L^\infty(0,T;\mathbb{V}),
\\
|\mathbf{D}(\widetilde{\mathbf{v}}_n)|^{p-2}\mathbf{D}(\widetilde{\mathbf{v}}_n) \xrightharpoonup[n\to\infty]{} |\mathbf{D}(\mathbf{v})|^{p-2} \mathbf{D}(\mathbf{v}) \quad
&\text{weakly in } L^{\frac{p}{p-1}}(0,T;\widetilde{\mathbb{L}}^{\frac{p}{p-1}}),
\\
\widetilde{\varphi}_n \xrightharpoonup[n\to\infty]{} \varphi\quad
&\text{weakly in } L^{2}(0,T),
\end{align}
\end{subequations}
where the pair $(\mathbf{v},\varphi)$ is a weak solution  to the following inverse problem:
\begin{equation}
\label{eq3.6}
\begin{cases}
&\begin{split} &\partial_{t} \mathbf{v}-\mu \ \mathrm{div}\big(|\mathbf{D}(\mathbf{v})|^{p-2}\mathbf{D}(\mathbf{v})\big)+
\mathrm{div}(\mathbf{v}\otimes\mathbf{v})-\varkappa\Delta \partial_{t}\mathbf{v}
+\nabla\pi
\\
&\qquad=\varphi(t)\,(\boldsymbol{\omega}({x})-\varkappa\Delta\boldsymbol{\omega}({x})) \ \ \text{ in }Q_T,
\end{split}
\\
&\operatorname{div}\mathbf{v}=0 \ \ \text{ in }Q_T,
\\
&\int_\Omega \mathbf{v}({x},t)\cdot (\boldsymbol{\omega}({x})-\varkappa\Delta\boldsymbol{\omega}({x}))\,d{x}=j(t) \ \ \text{ in }(0,T),
\\
&\mathbf{v}({x},0)=\overline{\mathbf{v}}({x},1) \ \ \text{ in }\Omega,
\\
&\mathbf{v}({x},t)=\boldsymbol{0} \ \ \text{ on }\Gamma_T,
\end{cases}
\end{equation}
with a new initial velocity $\mathbf{v}({x},0)$, which is a final velocity $\overline{\mathbf{v}}({x},1)$ of problem \eqref{eq3.4}. In general case, see \eqref{eq3.5}, it follows that $\mathbf{v}({x},0)\neq\mathbf{v}_0(x)$.
\item[{\bf 3.}] The family of functions  {$\displaystyle\{\mathbf{v}_n\}_{n\geq n_0}$} is relatively compact in  $L^2(0,T;\mathbb{H})$, and  {$\{\mathbf{1}_{(1/n,T)}\varphi_n\}_{n\geq n_0}$ is relatively compact in $L^{2}(0,T)$}. Therefore, there exist subsequences of $\displaystyle\{\mathbf{v}_n\}_{n\geq n_0}$ and $\{\mathbf{1}_{(1/n,T)}\varphi_n\}_{n\geq n_0}$, still labeled by $n$, such that
\begin{subequations}
\label{eq27}
\begin{align}
\mathbf{v}_n  \xrightarrow[n\to\infty]{} \mathbf{v}\quad
&\text{strongly in } L^2(0,T;\mathbb{H}),
\\
\mathbf{v}_n \xrightharpoonup[n\to\infty]{*} \mathbf{v}\quad
& \text{weakly$^*$ in } L^\infty(0,T;\mathbb{V}),
\\
|\mathbf{D}({\mathbf{v}}_n)|^{p-2}\mathbf{D}({\mathbf{v}}_n) \xrightharpoonup[n\to\infty]{} |\mathbf{D}(\mathbf{v})|^{p-2} \mathbf{D}(\mathbf{v}) \quad
&\text{weakly in } L^{\frac{p}{p-1}}(0,T;\widetilde{\mathbb{L}}^{\frac{p}{p-1}}),
\\
{\mathbf{1}_{(\frac1n,T)}\varphi_n \xrightharpoonup[n\to\infty]{} \varphi}\quad
&\text{weakly in } L^{2}(0,T),
\end{align}
\end{subequations}
where the pair $(\mathbf{v},\varphi)$ is a weak solution  to problem \eqref{eq3.6}.
\end{itemize}
\end{theorem}
\begin{remark}
Problem \eqref{eq3.1} is inverse, but the limit problem \eqref{eq3.4} is direct. This is explained by the fact that the term \eqref{sphi_n} in the limit as $n\to\infty$ does not depend on an unknown solution.
\end{remark}
The proof of this theorem is split into Secs. \ref{sec3.1}--\ref{sec3.3}.

\subsubsection*{\bf Scaling and shifting}
\vspace{2mm}
Following the carefully developed procedure outlined previously (see, for example, \cite{AKS-2024,AKSS-2024,AKSS-2024-2,KS-NARWA-2022,KS-JEPE-2022,KS-NA-TMA-2023}), let us consider some preliminary aspects before addressing the limiting passage as $n \to +\infty$ in \eqref{eq2.1}.

Assuming that the test vector function $\boldsymbol{\eta}$ in the equation \eqref{eq2.04} vanishes in a neighborhood of the subsection $\{t=T\}$, let us integrate the first summand on the left-hand side of \eqref{eq2.04} in $t$ by parts and write the resulting equality in an expanded form, separating the integrals over segments $(0,1/n)$ and $(1/n,T_n)$ from each other:
\begin{align}
& \int_0^{1/n} \int_\Omega \Big(-\mathbf{v}_n \cdot \partial_t \boldsymbol{\eta} - (\mathbf{v}_n \otimes \mathbf{v}_n) : \nabla \boldsymbol{\eta} + \mu \,
|\mathbf{D}(\mathbf{v}_n)|^{p-2}\mathbf{D}(\mathbf{v}_n):
\nabla \boldsymbol{\eta} \nonumber 
\\
& 
\quad - \varkappa \nabla \mathbf{v}_n : \nabla \partial_t\boldsymbol{\eta}-
\varphi_n(t)(\boldsymbol{\omega}({x}) -\varkappa\Delta\boldsymbol{\omega}({x})) \Big)\, d{x} dt \nonumber\\& \qquad - \int_\Omega \mathbf{v}_0({x}) \cdot \boldsymbol{\eta}({x},0)\, d{x} - \varkappa\int_\Omega \nabla\mathbf{v}_0({x}) : \nabla\boldsymbol{\eta}({x},0)\, d{x}  \nonumber 
\\
& + \int_{1/n}^{T_n} \int_\Omega \Big(-\mathbf{v}_n \cdot \partial_t \boldsymbol{\eta} - (\mathbf{v}_n \otimes \mathbf{v}_n) : \nabla \boldsymbol{\eta} + 
\mu|\mathbf{D}(\mathbf{v}_n)|^{p-2}\mathbf{D}(\mathbf{v}_n): \nabla \boldsymbol{\eta} \nonumber 
\\
& 
\quad - \varkappa \nabla \mathbf{v}_n : \nabla \partial_t\boldsymbol{\eta} -\varphi_n(t)(\boldsymbol{\omega}({x}) -\varkappa\Delta\boldsymbol{\omega}({x}))\Big)\, d{x} dt =0. \label{eq5.1}
\end{align}
Thus, \eqref{eq5.1} takes the form
\begin{align}
& \int_{Q_1}
\Big(-\overline{\mathbf{v}}_n({x},{\overline{t}}) \cdot \partial_{\overline{t}} \boldsymbol{\eta}({x},n^{-1}{\overline{t}})- \frac{1}{n} (\overline{\mathbf{v}}_n({x},{\overline{t}}) \otimes \overline{\mathbf{v}}_n({x},{\overline{t}})) : \nabla \boldsymbol{\eta}({x},n^{-1}{\overline{t}}) \nonumber 
\\[0.6ex]
&+\frac{\mu}{n}|\mathbf{D}(\overline{\mathbf{v}}_n) ({x},{\overline{t}})|^{p-2}\mathbf{D}(\overline{\mathbf{v}}_n )({x},{\overline{t}}) : \nabla \boldsymbol{\eta}({x},n^{-1}{\overline{t}})-\varkappa\nabla \overline{\mathbf{v}}_n ({x},{\overline{t}}) : \nabla \partial_{\overline{t}}\boldsymbol{\eta}({x},n^{-1}{\overline{t}})\nonumber 
\\[0.6ex]
&-\overline{\varphi}_n({\overline{t}})(\boldsymbol{\omega}({x})-\varkappa\Delta\boldsymbol{\omega}({x}))\cdot \boldsymbol{\eta}({x},n^{-1}{\overline{t}})\Big)\, d{x} d{\overline{t}} \nonumber 
\\[0.6ex]
&-\int_\Omega \mathbf{v}_0({x}) \cdot \boldsymbol{\eta}({x},0)\, d{x} -\varkappa \int_\Omega \nabla\mathbf{v}_0({x}) : \nabla\boldsymbol{\eta}({x},0)\, d{x}  
\nonumber 
\\[0.6ex]
& +\int_{Q_T}
\Big(-\widetilde{\mathbf{v}}_n ({x},\widetilde{t}) \cdot \partial_{\widetilde{t}}\, \boldsymbol{\eta}({x},\widetilde{t}+1/n)- (\widetilde{\mathbf{v}}_n({x},\widetilde{t}) \otimes \widetilde{\mathbf{v}}_n({x},\widetilde{t})):\nabla \boldsymbol{\eta}({x},\widetilde{t}+1/n) \nonumber 
\\[0.6ex]
& +\mu |\mathbf{D}(\widetilde{\mathbf{v}}_n) ({x},\widetilde{t})|^{p-2}\mathbf{D}(\widetilde{\mathbf{v}}_n )({x},\widetilde{t}) : \nabla \boldsymbol{\eta}({x},\widetilde{t} +1/n)-\varkappa\nabla \widetilde{\mathbf{v}}_n ({x},\widetilde{t}) : \nabla \partial_{\widetilde{t}}\,\boldsymbol{\eta}({x},\widetilde{t} +1/n) 
\nonumber 
\\[0.6ex]
&-\widetilde{\varphi}_n(\widetilde{t})(\boldsymbol{\omega}({x})-\varkappa\Delta\boldsymbol{\omega}({x}))\cdot \boldsymbol{\eta}({x},\widetilde{t} +1/n)
\Big)\, d{x} d\widetilde{t}=0. 
\label{eq5.4}
\end{align}
In accord with the further limiting passage as $n \to +\infty$ in \eqref{eq5.4}, we take the test vector-function $\boldsymbol{\eta}=\boldsymbol{\eta}_n({x},t)$ depending on $n$ as follows:
\begin{equation}
\label{eq5.5}
\boldsymbol{\eta}_n({x},t)=\left\{
\begin{array}{ll}
\displaystyle \overline{\boldsymbol{\eta}}({x},{\overline{t}}) \equiv \overline{\boldsymbol{\eta}}({x},nt) & \text{ for } t\in [0,1/n], \text{ i.e., for } {\overline{t}} \in [0,1], 
\\[2ex]
\widetilde{\boldsymbol{\eta}}({x},\widetilde{t}) \equiv \widetilde{\boldsymbol{\eta}}({x},t-1/n) & \text{ for } t\in [1/n,T_n],   \text{ i.e., for } \widetilde{t} \in [0,T],
\end{array}
\right.
\end{equation}
where $\overline{\boldsymbol{\eta}} = \overline{\boldsymbol{\eta}}({x},{\overline{t}})$ and $\widetilde{\boldsymbol{\eta}}=\widetilde{\boldsymbol{\eta}}({x},\widetilde{t})$ are arbitrary smooth test vector-functions defined on $\overline{\Omega} \times [0,1]$ and $\overline{\Omega} \times [0,T]$, respectively, such that $\overline{\boldsymbol{\eta}} = \widetilde{\boldsymbol{\eta}} \equiv 0$ holds in a neighbourhood of $\partial \Omega$, $\widetilde{\boldsymbol{\eta}} \equiv 0$ holds in a neighbourhood of the plane $\{\widetilde{t}=T\}$, and the matching condition
\begin{equation}
\label{eq5.6}
\overline{\boldsymbol{\eta}}({x},1^{-}) = \widetilde{\boldsymbol{\eta}}({x},0^+)
\end{equation}
holds.  We notice that condition \eqref{eq5.6} yields the weak derivative $\partial_t \boldsymbol{\eta}_n \in L^2(0,T_n;\mathbb{H})$, which implies $\boldsymbol{\eta}_n \! \in C([0,T_n];\mathbb{V})$, see \cite[Chapter 5, Section 5.9]{Evans}. Since 
\begin{align*}
\partial_{\overline{t}} \overline{\boldsymbol{\eta}}_n \in L^2(0,1;\mathbb{H})  \quad \text{and} \quad
\partial_{\widetilde{t}} \widetilde{\boldsymbol{\eta}}_n \in L^2(0,T;\mathbb{H}),   
\end{align*}
it follows, by combining the two time intervals, that 
$\partial_t \boldsymbol{\eta}_n \in L^2(0,T_n;\mathbb{H})$. Similarly, we  have
\begin{align*}
\overline{\boldsymbol{\eta}}_n \in L^2(0,1;\mathbb{V}^2)  \quad \text{and} \quad
\widetilde{\boldsymbol{\eta}}_n \in L^2(0,T;\mathbb{V}^2). 
\end{align*}
Hence, by combining the two time intervals,  we obtain
\begin{align*}
\boldsymbol{\eta}_n \in L^2(0,T_n;\mathbb{V}^2).
\end{align*}
Moreover, in view of the matching condition \eqref{eq5.6}, the one-sided traces are understood in the $\mathbb{H}$-sense, and there is no discontinuity in $\boldsymbol{\eta}_n$  at $t=\frac1n$, which corresponds to $\overline{t}=1$ and $\widetilde{t}=0$. Therefore, $\boldsymbol{\eta}_n$ is an admissible test vector function for \eqref{eq5.1} and \eqref{eq5.4}. 
Inserting \eqref{eq5.5} into \eqref{eq5.4}, we get
\begin{align}
& \int_{Q_1} \Big(-\overline{\mathbf{v}}_n({x},{\overline{t}}) \cdot \partial_{\overline{t}} \overline{\boldsymbol{\eta}}({x},{\overline{t}}) - \frac{1}{n} (\overline{\mathbf{v}}_n({x},{\overline{t}}) \otimes \overline{\mathbf{v}}_n({x},{\overline{t}})) : \nabla \overline{\boldsymbol{\eta}}({x},{\overline{t}})  \nonumber
\\[0.8ex]
& \quad+ \frac{\mu}{n}  \big(|\mathbf{D}(\overline{\mathbf{v}}_n) ({x},{\overline{t}})|^{p-2}\mathbf{D}(\overline{\mathbf{v}}_n )({x},{\overline{t}}) \big): \nabla \overline{\boldsymbol{\eta}}({x},{\overline{t}})- \varkappa\nabla \overline{\mathbf{v}}_n ({x},{\overline{t}}) : \nabla \partial_{\overline{t}}\overline{\boldsymbol{\eta}}({x},{\overline{t}}) \nonumber 
\\[0.8ex]
& \qquad-\overline{\varphi}_n({\overline{t}})(\boldsymbol{\omega}({x})-\varkappa\Delta\boldsymbol{\omega}(x))\cdot\overline{\boldsymbol{\eta}}({x},{\overline{t}}) \Big)\,d{x}d{\overline{t}}\nonumber 
\\[0.8ex]
& \qquad \quad-\int_\Omega \mathbf{v}_0({x}) \cdot \overline{\boldsymbol{\eta}}({x},0)\, d{x} -\varkappa\int_\Omega \nabla\mathbf{v}_0({x}) : \nabla \overline{\boldsymbol{\eta}}({x},0)\, d{x}  
\nonumber 
\\[0.8ex]
& +\int_{Q_T} \Big(-\widetilde{\mathbf{v}}_n ({x},\widetilde{t}) \cdot \partial_{\widetilde{t}}\, \widetilde{\boldsymbol{\eta}}({x},\widetilde{t}) - (\widetilde{\mathbf{v}}_n({x},\widetilde{t}) \otimes \widetilde{\mathbf{v}}_n({x},\widetilde{t})):\nabla \widetilde{\boldsymbol{\eta}}({x},\widetilde{t})\nonumber 
\\[0.8ex]
& \quad +\mu \, |\mathbf{D}(\widetilde{\mathbf{v}}_n) ({x},{\widetilde{t}})|^{p-2}\mathbf{D}(\widetilde{\mathbf{v}}_n )({x},{\widetilde{t}}) : \nabla \widetilde{\boldsymbol{\eta}}({x},\widetilde{t}) -\varkappa\nabla \widetilde{\mathbf{v}}_n ({x},\widetilde{t}) : \nabla \partial_{\widetilde{t}}\,\widetilde{\boldsymbol{\eta}}({x},\widetilde{t})\nonumber 
\\[0.8ex]
& \qquad 
-\widetilde{\varphi}_n(\widetilde{t})(\boldsymbol{\omega} ({x})-\varkappa\Delta\boldsymbol{\omega}(x))\cdot\widetilde{\boldsymbol{\eta}}({x},\widetilde{t})\Big)\, d{x} d\widetilde{t} =0.   \label{eq5.7}
\end{align}
Furthermore, we notice that
\begin{equation} \label{eq5.8}
\overline{\mathbf{v}}_n({x},1^-) = \widetilde{\mathbf{v}}_n({x},0^+) \quad \mbox{ in } \Omega
\end{equation}
in the sense of distributions, owing to the regularity properties of $\overline{\mathbf{v}}_n\in C([0,1],\mathbb{H})$ and $\widetilde{\mathbf{v}}_n\in C([0,T];\mathbb{H})$. The remainder of the proof of Theorem \ref{theorem3.1} is based  on a thorough study of \eqref{eq5.7}, taking into account \eqref{eq5.8}.

\subsection{\bf Proof of item 1 in Theorem \ref{theorem3.1}}
\label{sec3.1} 
\vspace{2mm}
We need to introduce a weak solution on  an infinitesimal initial layer.
\begin{definition}
\label{def3.1} A pair of functions $(\overline{\mathbf{v}}_n,\overline{\varphi}_n)$ is said to be a \textbf{weak solution} of the nonlinear inverse problem
\eqref{eq3.1} with initial condition $\mathbf{v}_0 \in \mathbb{V} \cap \mathbb{V}_p$ if
\begin{align*}
 \overline{\mathbf{v}}_n \in L^\infty(0,1; \mathbb{H} )\cap  L^\infty(0,1; \mathbb{V} )\cap L^p(0,1;\mathbb{V}_p),
\quad  \partial_{\overline{t}}\overline{\mathbf{v}}_n
\in L^{2}(0,1;\mathbb{V}), \quad
\overline{\varphi}_n\in L^{2}(0,1),
\end{align*}
and $(\overline{\mathbf{v}}_n,\overline{\varphi}_n)$  satisfies 
\begin{align*}
& \frac{d}{d{\overline{t}}}\int_\Omega
\Big(\overline{\mathbf{v}}_n({x},{\overline{t}})\cdot\boldsymbol{\phi}({x})+\varkappa\nabla\overline{\mathbf{v}}_n({x},{\overline{t}}):\nabla\boldsymbol{\phi}({x})\Big)\,d{x}
+\frac{\mu}{n}\int_\Omega|\mathbf{D}(\overline{\mathbf{v}}_n) ({x},{\overline{t}})|^{p-2}\mathbf{D}(\overline{\mathbf{v}}_n )({x},{\overline{t}}):\nabla\boldsymbol{\phi}({x})\,d{x}
\\
& -\frac1n\int_\Omega
(\overline{\mathbf{v}}_n({x},{\overline{t}})\cdot\nabla)\boldsymbol{\phi}({x})\cdot\overline{\mathbf{v}}_n({x},{\overline{t}})\,d{x} =\overline{\varphi}_n({\overline{t}})\int_\Omega(\boldsymbol{\omega}({x})-\varkappa\Delta \boldsymbol{\omega}(x))\cdot \boldsymbol{\phi}({x})\,d{x} \quad \text{for a.e. }{\overline{t}}\in(0,1),
\end{align*}
for every test function $\boldsymbol{\phi}\in\mathbb{V}$ and the corresponding compatibility condition
\begin{align*}
&\int_\Omega\mathbf{v}_0({x})\cdot(\boldsymbol{\omega}({x})-\varkappa\Delta \boldsymbol{\omega}({x}))\,d{x}=\overline{j}_n(0).
\end{align*}
\end{definition}
\begin{proposition}[Estimate for a scaled solution]
\label{proposition.2.1} Under Conditions \ref{C}, the problem \eqref{eq3.1} has {a unique weak solution} in the sense of Definition \ref{def3.1}, which satisfies the estimate
\begin{align}
&\underset{{\overline{t}}\in(0,1)}{\rm esssup} \Big (
\|\overline{\mathbf{v}}_n(\cdot,{\overline{t}})\|_{\mathbb{H}}^2 +\|\nabla \overline{\mathbf{v}}_n(\cdot,{\overline{t}})\|_{\mathbb{H}}^2 \Big)
+\frac{1}{n}\|\overline{\mathbf{v}}_n\|_{L^p(0,1;\mathbb{V}_p)}^p
\nonumber
\\
&\quad +
\|\partial_{\overline{t}}\overline{\mathbf{v}}_n\|_{L^2(0,1;\mathbb{V})}^2 
+\|\overline{\varphi}_n\|_{L^{2}(0,1)}\leq C,
\label{eq:3est-global-1}
\end{align}
where the constant $C$ does not depend on $n$. 
\end{proposition}
\noindent{\bf Proof}. We apply Theorem \ref{Theorem.2.1}, and estimate \eqref{eq:first-int} yields
\begin{align}
\label{eq3.15}
&\underset{0 \leq {\overline{t}} \leq 1}{\rm esssup}  \Big( \|\overline{\mathbf{v}}_n(\cdot,{\overline{t}})\|_{\mathbb{H}}^2+\| \nabla \overline{\mathbf{v}}_n(\cdot,{\overline{t}})\|_{\mathbb{H}}^2 \Big)+\frac{1}{n}
\|\overline{\mathbf{v}}_n\|_{L^p(0,1;\mathbb{V}_p)}^p
\nonumber
\\
&
\leq \underset{0 \leq {t} \leq T_n }{\rm esssup}\Big(\|\mathbf{v}_n(\cdot,t)\|_{\mathbb{H}}^2+\|\nabla\mathbf{v}_n(\cdot,t)\|_{\mathbb{H}}^2\Big)+\|\mathbf{v}_n\|_{L^p(0,T_n;\mathbb{V}_p)}^p
\leq C,
\end{align}
and from these estimate, we obtain the uniform estimate $\overline{\varphi}_n\in L^{2}(0,1)$:
\begin{align*}
&\|\overline{\varphi}_n\|_{L^{2}(0,1)}=
\Big\|\overline{j}_n^{\prime}
(\cdot)+\frac{\mu}{n}\int_\Omega |\mathbf{D}\overline{\mathbf{v}}_n({x},\cdot)|^{p-2} \mathbf{D}\overline{\mathbf{v}}_n({x},\cdot)
:\nabla\boldsymbol{\omega}({x})\,d{x} \nonumber
\\
&\quad
-\frac1n\int_\Omega
(\overline{\mathbf{v}}_n({x},\cdot)\cdot\nabla)\boldsymbol{\omega}({x})\cdot\overline{\mathbf{v}}_n({x},\cdot)\,d{x}
\Big\|_{L^{2}(0,1)}
\\
&\leq
\|\overline{j}_n'\|_{L^{2}(0,1)}
+\frac{\mu}{n}
\Big\|\int_\Omega |\mathbf{D}\overline{\mathbf{v}}_n({x},\cdot)|^{p-2} \mathbf{D}\overline{\mathbf{v}}_n({x},\cdot)
:\nabla\boldsymbol{\omega}({x})\,d{x}\Big\|_{L^{2}(0,1)}
\\
&\quad
+\frac1n
\Big\|\int_\Omega
(\overline{\mathbf{v}}_n({x},\cdot)\cdot\nabla)\boldsymbol{\omega}({x})\cdot\overline{\mathbf{v}}_n({x},\cdot)\,d{x}\Big\|_{L^{2}(0,1)}\leq C,
\end{align*}
where $C$ does not depend on $n$. In order to estimate ${\partial_{\overline{t}}\overline{\mathbf{v}}_n}$ in $L^{2}(0,1;\mathbb{V})$, a calculation similar to \eqref{eq3.11} yields
\begin{align}
&\int_\Omega \Big(|\partial_{\overline{t}} \overline{\mathbf{v}}_n|^2
+\varkappa|\nabla\partial_{\overline{t}} \overline{\mathbf{v}}_n|^2 \Big)\,dx
\leq \frac{\varkappa}{4} \|\nabla\partial_{\overline{t}} \overline{\mathbf{v}}_n \|_{\mathbb{H}}^2 + C\| \overline{\mathbf{v}}_n \|_{\mathbb{H}}^{4(1-\frac{d}4)} \|\nabla \overline{\mathbf{v}}_n \|_{\mathbb{H}}^{d} 
\nonumber
\\
& 
\qquad+ \frac{\mu^2}{\varkappa}\| \mathbf{D}( \overline{\mathbf{v}}_n) \|_{\widetilde{\mathbb{L}}^{2(p-1)}}^{2(p-1)}
+\frac{\varkappa}{4}\|\nabla\partial_{\overline{t}} \overline{\mathbf{v}}_n \|_{\mathbb{H}}^2 
+\frac{1}{2}|\overline{\varphi}_n(\overline{t})|^2 + \frac{1}{2}|\overline{j}_n'(\overline{t})|^2,
\label{eq3.11'}
\end{align}
and integrate over $\overline{t}\in(0,1)$:
\begin{align}
\label{E9''}
&\|\partial_{\overline{t}}\overline{\mathbf{v}}_n \|_{L^{2}(0,1;\mathbb{H})}^2 
+\frac{\varkappa}{2} \, \|\nabla\partial_{\overline{t}}\overline{\mathbf{v}}_n \|_{L^{2}(0,1;\mathbb{H})}^2 
\leq C \|\overline{\mathbf{v}}_n \|_{L^\infty(0,1;\mathbb{H})}^{4-d}  \|\nabla\overline{\mathbf{v}}_n \|_{L^\infty(0,1;\mathbb{H})}^d  
\nonumber
\\
& \quad+ \frac{\mu^2}{\varkappa} \int_{0}^{1}\| \mathbf{D}(\overline{\mathbf{v}}_n) (\cdot,\overline{t})\|_{\widetilde{\mathbb{L}}^{2(p-1)}}^{2(p-1)}\,d\overline{t} +\frac{1}{2} \int_{0}^{1}|\overline{\varphi}_n(\overline{t})|^2 \,d\overline{t}+\frac{1}{2} \int_0^1|\overline{j}_n'(\overline{t})|^2 \,d\overline{t},
\end{align}
which leads to ${\partial_{\overline{t}}\overline{\mathbf{v}}_n} \in L^{2}(0,1;\mathbb{V})$ for $p \in (1,2]$. \null\nobreak\hfill\ensuremath{\square}

To complete the proof of item 1 of Theorem \ref{theorem3.1}, we must prove \eqref{eq3.3a} and \eqref{eq3.3b}.
The integral equi-continuity of $\{\overline{\varphi}_n\}_{n\geq n_0}$ in $L^{2}(0,1)$ is easily established.
The convergence \eqref{eq3.3a} follows from the representation
\begin{align*}
{\overline{\varphi}_n({\overline{t}})-
\overline{j}'({\overline{t}})}&=
\frac{\mu}{n}\int_\Omega |\mathbf{D}\overline{\mathbf{v}}_n({x},{\overline{t}})|^{p-2} \mathbf{D}\overline{\mathbf{v}}_n({x},{\overline{t}}):\nabla\boldsymbol{\omega}({x})\,d{x}
\\
&\quad
-\frac1n\int_\Omega
(\overline{\mathbf{v}}_n({x},{\overline{t}})\cdot\nabla)\boldsymbol{\omega}({x})\cdot\overline{\mathbf{v}}_n({x},{\overline{t}})\,d{x} +\overline{j}_n'({\overline{t}})-\overline{j}'({\overline{t}}).
\end{align*}
Applying the H\"older inequality, we need to prove the convergence of $\{\overline{\varphi}_n\}_{n\geq n_0}$ only in $L^2(0,1)$.
We deduce that
\begin{align}\label{Equicon}
&\left(\int_0^1|\overline{\varphi}_n({\overline{t}})-\overline{j}'({\overline{t}})|^2\,d{\overline{t}}\right)^{\frac12}
\leq
\frac{\mu}{n^{\frac1p}}\left(\frac1n\int_{Q_1}
|\mathbf{D}\overline{\mathbf{v}}_n({x},{\overline{t}})|^p\,d{x}d{\overline{t}}\right)^{\frac{p-1}p}
\left(\int_\Omega|\nabla\boldsymbol\omega (x)|^p\,d{x}\right)^{\frac1p}
\nonumber
\\
&
\quad+\frac1n \|\nabla\boldsymbol\omega \|_{\widetilde{\mathbb{L}}^\infty}\underset{{\overline{t}}\in(0,1)}{\rm esssup}
\int_\Omega|\overline{\mathbf{v}}_n({x},{\overline{t}})|^2\,d{x} 
+\left(\int_0^1|\overline{j}_n'({\overline{t}})
-\overline{j}'({\overline{t}})|^2\,d{\overline{t}}\right)^{\frac12}.
\end{align}
Using condition ${\bf C}_{7}$ in \eqref{Equicon}, we infer that
\begin{align*}
\left(\int_0^1|\overline{\varphi}_n({\overline{t}})-\overline{j}'({\overline{t}})|^2\,d{\overline{t}}\right)^{\frac12}
\longrightarrow 0  \ \ \text{as} \ n \longrightarrow \infty.
\end{align*}
This leads to $\{\overline{\varphi}_n\}_{n\geq n_0}$ converging strongly in $L^{2}(0,1)$ to $\overline{j}'({\overline{t}})$.
In order to prove \eqref{eq3.3b}, we need to show that the family $\{\overline{\mathbf{v}}_n\}_{n\geq n_0}$  is relatively compact in $C([0,1];\mathbb{H})$. This follows from the Aubin--Lions--Simon lemma in $C^0$ \cite[Th. A.52]{Bedrossian}, we prove equi-continuity:
\begin{align*}
&\int_\Omega|\overline{\mathbf{v}}_n(x,\overline{t}+h)-\overline{\mathbf{v}}_n(x,\overline{t})
|^2\,dx=\int_\Omega
\int_{\overline{t}}^{\overline{t}+h}
\partial_\vartheta\overline{\mathbf{v}}_n(x,\vartheta)\,d\vartheta\cdot\big(\overline{\mathbf{v}}_n(x,\overline{t}+h)-\overline{\mathbf{v}}_n(x,\overline{t})\big)\,dx
\\
&\leq 
\int_\Omega
\int_{\overline{t}}^{\overline{t}+h}
\vert\partial_\vartheta\overline{\mathbf{v}}_n(x,\vartheta)\vert\,d\vartheta\cdot(
\vert\overline{\mathbf{v}}_n(x,\overline{t}+h)\vert+\vert\overline{\mathbf{v}}_n(x,\overline{t})
\vert)\,dx
\\
&\leq 2\int_{\overline{t}}^{\overline{t}+h}
\left(\int_\Omega
\vert\partial_\vartheta\overline{\mathbf{v}}_n(x,\vartheta)\vert^2\,dx\right)^{\frac12}
\,d\vartheta
\left(\underset{{\overline{t}}\in(0,1)}{\rm esssup}\int_\Omega\vert\overline{\mathbf{v}}_n(x,\overline{t})\vert^2\,dx\right)^{\frac12}
\\
&\leq 
2\left(\int_{Q_1}\vert\partial_\vartheta\overline{\mathbf{v}}_n(x,\vartheta)\vert^2\,dxd\vartheta\right)^{\frac{1}{2}}
h^{\frac1{2}}
\left(\underset{{\overline{t}}\in(0,1)}{\rm esssup}\int_\Omega\vert\overline{\mathbf{v}}_n(x,\overline{t})\vert^2\,dx\right)^{\frac12}
\leq C h^{\frac1{2}}.
\end{align*}
 Therefore, there exist subsequences of $\{\overline{\mathbf{v}}_n\}_{n\geq n_0}$, still labeled by $n$, which converges weakly to a function. We need to show that this function is $\overline{\mathbf{v}}({x},{\overline{t}})$ satisfies problem \eqref{eq3.4}.

We put in \eqref{eq5.7} the test function $\widetilde{\boldsymbol{\eta}}\equiv0$ in the domain $Q_T$, and take $\overline{\boldsymbol{\eta}}\theta_k(1-\overline{t})$, where $\theta_k(s)=sk\mathbf{1}_{(0,1/k]}(s)+
\mathbf{1}_{(1/k,1]}(s)$ approximates the Heaviside function. Here, we assume that $\overline{\boldsymbol{\eta}}$ has non-zero traces at $\overline{t}=1$.  These choices of the test functions lead to
\begin{align*}
& \int_{Q_1} \theta_k(1-\overline{t})\Big(-\overline{\mathbf{v}}_n({x},{\overline{t}}) \cdot \partial_{\overline{t}} \overline{\boldsymbol{\eta}}({x},{\overline{t}})
-\varkappa\nabla\overline{\mathbf{v}}_n({x},{\overline{t}}):\nabla \partial_{\overline{t}} \overline{\boldsymbol{\eta}}({x},{\overline{t}})
- \frac{1}{n} (\overline{\mathbf{v}}_n({x},{\overline{t}}) \otimes \overline{\mathbf{v}}_n({x},{\overline{t}})) : \nabla\overline{\boldsymbol{\eta}}({x},{\overline{t}})
\\
& \quad+ \frac{\mu}{ n}
|\mathbf{D}\overline{\mathbf{v}}_n ({x},{\overline{t}})|^{p-2}
\mathbf{D}\overline{\mathbf{v}}_n ({x},{\overline{t}}) : \nabla \overline{\boldsymbol{\eta}}({x},{\overline{t}})-\overline{\varphi}_n({\overline{t}})(\boldsymbol{\omega}({x})-\varkappa\Delta \boldsymbol{\omega}({x}))\cdot\overline{\boldsymbol{\eta}}({x},{\overline{t}}) \Big)\,d{x}d{\overline{t}} 
\\
& \quad-\int_\Omega \mathbf{v}_0({x}) \cdot \overline{\boldsymbol{\eta}}({x},0)\, d{x}
-\varkappa\int_\Omega \nabla\mathbf{v}_0({x}):\nabla\overline{\boldsymbol{\eta}}({x},0)\, d{x}
\\
& \quad +k\int_{1-\frac1k}^1\int_\Omega \overline{\mathbf{v}}_n({x},\overline{t}) \cdot \overline{\boldsymbol{\eta}}({x},\overline{t})\,d{x}d\overline{t} +k\int_{1-\frac1k}^1\int_\Omega \nabla\overline{\mathbf{v}}_n({x},\overline{t}) : \nabla\overline{\boldsymbol{\eta}}({x},\overline{t})\,d{x}d\overline{t}=0,
\end{align*}
with subsequent passing to the limit as $k\to\infty$:
\begin{align*}
& \int_{Q_1} \Big(-\overline{\mathbf{v}}_n({x},{\overline{t}}) \cdot \partial_{\overline{t}} \overline{\boldsymbol{\eta}}({x},{\overline{t}})
-\varkappa\nabla\overline{\mathbf{v}}_n({x},{\overline{t}}):\nabla \partial_{\overline{t}} \overline{\boldsymbol{\eta}}({x},{\overline{t}})
- \frac{1}{n} (\overline{\mathbf{v}}_n({x},{\overline{t}}) \otimes \overline{\mathbf{v}}_n({x},{\overline{t}})) : \nabla\overline{\boldsymbol{\eta}}({x},{\overline{t}})
\\
& \quad+ \frac{\mu}{ n}|\mathbf{D}\overline{\mathbf{v}}_n ({x},{\overline{t}})|^{p-2}
\mathbf{D}\overline{\mathbf{v}}_n ({x},{\overline{t}}) : \nabla \overline{\boldsymbol{\eta}}({x},{\overline{t}})-\overline{\varphi}_n({\overline{t}})(\boldsymbol{\omega}({x})-\varkappa\Delta \boldsymbol{\omega}({x}))\cdot\overline{\boldsymbol{\eta}}({x},{\overline{t}}) \Big)\,d{x}d{\overline{t}} 
\\
& \quad-\int_\Omega \mathbf{v}_0({x}) \cdot \overline{\boldsymbol{\eta}}({x},0)\, d{x}
-\varkappa\int_\Omega \nabla\mathbf{v}_0({x}):\nabla\overline{\boldsymbol{\eta}}({x},0)\, d{x}
\\
& \quad +\int_\Omega \overline{\mathbf{v}}_n({x},1) \cdot \overline{\boldsymbol{\eta}}({x},1)\,d{x} + \varkappa\int_\Omega \nabla\overline{\mathbf{v}}_n({x},1) : \nabla\overline{\boldsymbol{\eta}}({x},1)\,d{x}
=0.
\end{align*}
Let us accomplish the passage to the limit in problem \eqref{eq3.1}. Proposition \ref{proposition.2.1} involves 
\begin{align*}
&\left|
\frac{\mu}{n}\int_{Q_1}
|\mathbf{D}\overline{\mathbf{v}}_n|^{p-2}
\mathbf{D}\overline{\mathbf{v}}_n:\nabla\overline{\boldsymbol{\eta}}\,
d{x}d{\overline{t}}\right|
 \leq \frac{\mu}{n^{\frac{1}{p}}}
\left( \frac{1}{n}\int_{Q_1}|\mathbf{D}\overline{\mathbf{v}}_n|^{p}\,d{x}d{\overline{t}}\right)^{(p-1)/p}
{\|\overline{\boldsymbol{\eta}}\|_{L^p(0,1;\mathbb{V}_p)}}\longrightarrow 0,
\\
&\left|\frac{1}{n}\int_{Q_1}
(\overline{\mathbf{v}}_n\cdot\nabla)\overline{\boldsymbol{\eta}}\cdot \overline{\mathbf{v}}_n\,d{x}d{\overline{t}} \right|\leq
\frac{C}{n}
\underset{{\overline{t}}\in(0,1)}{\mathrm{esssup}}\|\overline{\mathbf{v}}_n(\cdot,{\overline{t}})\|_{\mathbb{H}}^{1-\frac{d}{4}}
\underset{{\overline{t}}\in(0,1)}{\mathrm{esssup}}\|\nabla\overline{\mathbf{v}}_n(\cdot,{\overline{t}})\|_{\mathbb{H}}^{\frac{d}{4}} 
\\
& \qquad \qquad \qquad \qquad \qquad \qquad \times{\|\overline{\boldsymbol{\eta}}\|_{L^2(0,1;\mathbb{V})}}\longrightarrow0,
\end{align*}
as $n\longrightarrow \infty$, {for all $\overline{\boldsymbol{\eta}}\in L^2(0,1;\mathbb{V}) \cap L^p(0,1;\mathbb{V}_p)$.} Therefore, in the limit, we obtain
\begin{align*}
& \int_{Q_1} \Big(-\overline{\mathbf{v}}({x},{\overline{t}}) \cdot \partial_{\overline{t}} \overline{\boldsymbol{\eta}}({x},{\overline{t}})
-\varkappa\nabla\overline{\mathbf{v}}({x},{\overline{t}}):\nabla\partial_{\overline{t}} \overline{\boldsymbol{\eta}}({x},{\overline{t}})
-\overline{\varphi}(\overline{t})(\boldsymbol{\omega}({x})-\varkappa\Delta \boldsymbol{\omega}({x}))\cdot\overline{\boldsymbol{\eta}}({x},{\overline{t}}) \Big)\,d{x}d{\overline{t}}
\\
&\quad-\int_\Omega \mathbf{v}_0({x}) \cdot \overline{\boldsymbol{\eta}}({x},0)\, d{x}
-\varkappa\int_\Omega\nabla\mathbf{v}_0({x}):\nabla\overline{\boldsymbol{\eta}}({x},0)\, d{x}
\\
& \quad +\int_\Omega \overline{\mathbf{v}}_n({x},1) \cdot \overline{\boldsymbol{\eta}}({x},1)\,d{x} + \varkappa\int_\Omega \nabla\overline{\mathbf{v}}({x},1) : \nabla\overline{\boldsymbol{\eta}}({x},1)\,d{x}
=0,
\end{align*}
where is a weak formulation of the problem \eqref{eq3.4} with $\overline{\varphi}(\overline{t})=\overline{j}^\prime(\overline{t})$.
 \null\nobreak\hfill\ensuremath{\square}

\subsection{\bf Proof of item 2 in Theorem \ref{theorem3.1}}
\label{sec3.2}
\vspace{2mm}
In this subsection, we formulate the definition of weak solutions to problem \eqref{eq3.2}.
\begin{definition}
\label{def3.2w} A pair of functions $(\widetilde{\mathbf{v}}_n,\widetilde{\varphi}_n)$ is said to be a \textbf{weak solution} of the nonlinear inverse problem \eqref{eq3.2} with initial condition $\overline{\mathbf{v}}_n({x},1) \in \mathbb{V} \cap \mathbb{V}_p$ if
\begin{align*}
&\widetilde{\mathbf{v}}_n\in {L^\infty(0,T;\mathbb{H})} \cap L^\infty(0,T;\mathbb{V})\cap L^p(0,T;\mathbb{V}_p),
\quad
\partial_{\widetilde{t}}\widetilde{\mathbf{v}}_n\in 
L^{2}(0,T;\mathbb{V}), \quad \widetilde{\varphi}_n\in L^{2}(0,T),
\end{align*}
and $(\widetilde{\mathbf{v}}_n,\widetilde{\varphi}_n)$  satisfies 
\begin{align*}
& \frac{d}{d\widetilde{t}}\int_\Omega
\Big(\widetilde{\mathbf{v}}_n({x},\widetilde{t})
\cdot\boldsymbol{\phi}({x})+\varkappa\nabla\widetilde{\mathbf{v}}_n({x},\widetilde{t})
:\nabla\boldsymbol{\phi}({x})\Big)\,d{x}
+\mu\int_\Omega|\mathbf{D}\widetilde{\mathbf{v}}_n({x},\widetilde{t})|^{p-2}\mathbf{D}\widetilde{\mathbf{v}}_n({x},\widetilde{t})
:\nabla\boldsymbol{\phi}({x})\,d{x} \\
& -\int_\Omega
(\widetilde{\mathbf{v}}_n({x},\widetilde{t})
\cdot\nabla)\boldsymbol{\phi}({x})\cdot\widetilde{\mathbf{v}}_n({x},\widetilde{t})
\,d{x}=\widetilde{\varphi}_n(\widetilde{t})\int_\Omega (\boldsymbol{\omega}({x})-\varkappa \Delta\boldsymbol{\omega}(x))\cdot\boldsymbol{\phi}({x})\,d{x} \, \, \, \text{ for a.e. }\widetilde{t}\in(0,T),
\end{align*}
for every test function
$\boldsymbol{\phi}\in\mathbb{V}$ and the corresponding compatibility condition
\begin{align*}
&\int_\Omega
\overline{\mathbf{v}}_n({x},1)\cdot(\boldsymbol{\omega}({x})-\varkappa \Delta\boldsymbol{\omega}(x))\,d{x}=\widetilde{j}_n(0).
\end{align*}
\end{definition}
\begin{proposition}[\textbf{Estimate for a shifted solution}]
\label{proposition.2.2} 
Under Conditions \ref{C}, the problem
\eqref{eq3.2} has a unique weak solution in the sense
of Definition \ref{def3.2w}, which satisfies the estimate
\begin{align}
&\underset{\widetilde{t}\in(0,T)}{\rm esssup} \Big(\|\widetilde{\mathbf{v}}_n(\cdot,\widetilde{t})
\|_{\mathbb{H}}^2+\|\nabla\widetilde{\mathbf{v}}_n(\cdot,\widetilde{t})
\|_{\mathbb{H}}^2 \Big)+
\|\widetilde{\mathbf{v}}_n\|_{L^p(0,T;\mathbb{V}_p)}^p 
+\|\widetilde{\varphi}_n\|_{L^{2}(0,T)}^2
\nonumber
\\
&\quad+
\|\partial_{\widetilde{t}}\widetilde{\mathbf{v}}_n\|_{L^{2}(0,T;\mathbb{V})}^2
\leq C,
\label{eq:est-global-11}
\end{align}
where the constant $C$ does not depend on $n$. 
\end{proposition}

\noindent{\bf Proof.}
An estimate similar to \eqref{eq:first-int} and \eqref{eq:second-int} gives
\begin{align*}
&\|\widetilde{\mathbf{v}}_n\|
{L^\infty(0,T;\mathbb{H})}^2+\|\widetilde{\mathbf{v}}_n\|_{L^\infty(0,T;\mathbb{V})}^2+
\|\widetilde{\mathbf{v}}_n\|_{L^p(0,T;\mathbb{V}_p)}^p
+\|\partial_{\widetilde{t}}\widetilde{\mathbf{v}}_n\|
{L^2(0,T;\mathbb{V})}^2+\|\widetilde{\varphi}n\|_{L^2(0,T)}^2\leq C,
\end{align*}
where $C$ does not depend on $n$,  as shown in Remark \ref{rem3.1}.
\null\nobreak\hfill\ensuremath{\square}

The integral equi-continuity in $\widetilde{t}$ for $\{\widetilde{\mathbf{v}}_n\}_{n\geq n_0}$ follows from Corollary \ref{cor3.11}. According to the Aubin--Lions lemma in $C^0$ version, see Lemma \ref{ALSlemma-1}, that  
$\{\widetilde{\mathbf{v}}_n\}_{n\geq n_0}$ is relatively compact in $L^2(0,T;\mathbb{H})$.
\begin{corollary}
\label{cor4.1}
The family $\{\widetilde{\mathbf{v}}_n\}_{n\geq n_0}$ is integral equicontinuous in time in $L^2(0,T;\mathbb{H})$.
\end{corollary}
\noindent{\bf Proof.} This follows from the equi-continuity, see Corollary \ref{cor3.11}:
\begin{align*}
&\int_0^{T-h}\int_\Omega
|\widetilde{\mathbf{v}}_n({x},
\widetilde{t}+h)-
\widetilde{\mathbf{v}}_n({x},\widetilde{t})|^2\,d{x}d\widetilde{t}+
\int_0^{T-h}\int_\Omega
|\nabla(\widetilde{\mathbf{v}}_n({x},
\widetilde{t}+h)-
\widetilde{\mathbf{v}}_n({x},\widetilde{t}))|^2\,d{x}d\widetilde{t}=
\\
&\quad=\int_{\frac1n}^{T_n-h}\int_\Omega
|\mathbf{v}_n({x},t+h)-
\mathbf{v}_n({x},t)|^2\,d{x}dt
\\
& \qquad+\int_{\frac1n}^{T_n-h}\int_\Omega
|\nabla(\mathbf{v}_n({x},t+h)-
\mathbf{v}_n({x},t))|^2\,d{x}dt\le C\omega(h),
\end{align*}
where $C$ does not depend on $n$ and $h$, and $\omega(h)$ is defined in Corollary \ref{cor3.11}.
\null\nobreak\hfill\ensuremath{\square}

\begin{corollary}
The family $\{\widetilde{\mathbf{v}}_n\}_{n\geq n_0}$ coincides in the limit with  $\{\mathbf{v}_n\}_{n\geq n_0}$ in $L^2(0,T;\mathbb{V})$.
\end{corollary}

\noindent{\bf Proof.}
This follows from Corollaries \ref{cor3.11} and \ref{cor4.1}.
\null\nobreak\hfill\ensuremath{\square}

\begin{corollary}
The family $\{\widetilde{\mathbf{v}}_n\}_{n\geq n_0}$ is relatively compact in $C([0,T];\mathbb{H})$.
\end{corollary}
\noindent{\bf Proof.} Since $\{\widetilde{\mathbf{v}}_n\}_{n\geq n_0}$ is uniformly bounded in $L^\infty(0,T;\mathbb{V})$, we get
\begin{align}
&\int_\Omega|\widetilde{\mathbf{v}}_n({x},\widetilde{t}+h)
-\widetilde{\mathbf{v}}_n({x},\widetilde{t})|^2\,d{x}
\nonumber
\\&
=\Big\vert
\int_{\widetilde{t}}^{\widetilde{t}+h}
\int_\Omega\partial_s\widetilde{\mathbf{v}}_n({x},s)\cdot
(\widetilde{\mathbf{v}}_n({x},\widetilde{t}+h)-\widetilde{\mathbf{v}}_n({x},\widetilde{t}))
\,d{x}ds\Big\vert\nonumber
\\
&
\leq \int_{\widetilde{t}}^{\widetilde{t}+h}
\|\widetilde{\mathbf{v}}_n(\cdot,\widetilde{t}+h)-\widetilde{\mathbf{v}}_n(\cdot,\widetilde{t})\|_{\mathbb{H}}
\|\partial_{s}\widetilde{\mathbf{v}}_n(\cdot,s)\|_{\mathbb{H}}\,ds
\nonumber
\\
&
\leq 
2 h^{\frac1{2}}\left(\int_0^T
\|\partial_{s}\widetilde{\mathbf{v}}_n(\cdot,s)\|_{\mathbb{H}}^2\,ds\right)^{\frac{1}{2}}
\underset{\widetilde{t}\in(0,T)}{\mathrm{esssup}}
\|\widetilde{\mathbf{v}}_n(\cdot,\widetilde{t})\|_{\mathbb{H}}\leq Ch^{\frac1{2}},
\label{eq3.23}
\end{align}
where $C$ does not depend on $n$ and $h$. Due to the Aubin--Lions lemma in $C^0$ version, see Lemma \ref{ALSlemma-1}, we prove that $\{\widetilde{\mathbf{v}}_n\}_{n\geq n_0}$ is relatively compact in 
$C([0,T];\mathbb{H})$.
\null\nobreak\hfill\ensuremath{\square}

We need to pass to the limit as $n\to\infty$. To this end, we put $\overline{\boldsymbol{\eta}}\equiv0$ in \eqref{eq5.7} and rearrange a test function $\widetilde{\boldsymbol{\eta}}$ on the plane $\widetilde{t}=0$ as the product $\widetilde{\boldsymbol{\eta}}$ and the approximation of the Heaviside function $\theta_k(\widetilde{t})=k\widetilde{t}\mathbf{1}_{[0,1/k]}(\widetilde{t})+\mathbf{1}_{(1/k,T)}(t)$:
\begin{align*}
&\int_{Q_T} \theta_k(\widetilde{t})\Big(-\widetilde{\mathbf{v}}_n ({x},\widetilde{t}) \cdot \partial_{\widetilde{t}}\, \widetilde{\boldsymbol{\eta}}({x},\widetilde{t})
-\varkappa\nabla\widetilde{\mathbf{v}}_n ({x},\widetilde{t}):\nabla \partial_{\widetilde{t}}\, \widetilde{\boldsymbol{\eta}}({x},\widetilde{t})
- (\widetilde{\mathbf{v}}_n({x},\widetilde{t}) \otimes \widetilde{\mathbf{v}}_n({x},\widetilde{t})):\nabla \widetilde{\boldsymbol{\eta}}({x},\widetilde{t}) 
\\[0.8ex]
&\quad
+\mu|\mathbf{D}\widetilde{\mathbf{v}}_n ({x},\widetilde{t})|^{p-2}\mathbf{D}\widetilde{\mathbf{v}}_n ({x},\widetilde{t}) : \nabla \widetilde{\boldsymbol{\eta}}({x},\widetilde{t})
-\widetilde{\varphi}_n(\widetilde{t})(\boldsymbol{\omega}({x})-\varkappa\Delta\boldsymbol{\omega}({x}))\cdot\widetilde{\boldsymbol{\eta}}({x},\widetilde{t})\Big)\, d{x} d\widetilde{t} 
\\[0.8ex]
&\quad
-k\int_0^{\frac1k}\int_\Omega\widetilde{\mathbf{v}}_n({x},\widetilde{t})
\cdot\widetilde{\boldsymbol{\eta}}({x},\widetilde{t})\,d{x}d\widetilde{t}
-\varkappa \, k\int_0^{\frac1k}\int_\Omega \nabla\widetilde{\mathbf{v}}_n({x},\widetilde{t})
: \nabla \widetilde{\boldsymbol{\eta}}({x},\widetilde{t})\,d{x}d\widetilde{t}
=0.
\end{align*}
Next, we pass to the limit as $k\to\infty$, and encompass the traces at $\widetilde{t}=0$ by the rules \eqref{eq5.8}:
\begin{align*}
&\int_{Q_T} \Big(-\widetilde{\mathbf{v}}_n ({x},\widetilde{t}) \cdot \partial_{\widetilde{t}}\, \widetilde{\boldsymbol{\eta}}({x},\widetilde{t})
-\varkappa\nabla\widetilde{\mathbf{v}}_n ({x},\widetilde{t}):\nabla \partial_{\widetilde{t}}\, \widetilde{\boldsymbol{\eta}}({x},\widetilde{t})
- (\widetilde{\mathbf{v}}_n({x},\widetilde{t}) \otimes \widetilde{\mathbf{v}}_n({x},\widetilde{t})):\nabla \widetilde{\boldsymbol{\eta}}({x},\widetilde{t}) 
\\[0.8ex]
&\quad
+\mu|\mathbf{D}\widetilde{\mathbf{v}}_n ({x},\widetilde{t})|^{p-2}\mathbf{D}\widetilde{\mathbf{v}}_n ({x},\widetilde{t}) : \nabla \widetilde{\boldsymbol{\eta}}({x},\widetilde{t})
-\widetilde{\varphi}_n(\widetilde{t})(\boldsymbol{\omega}({x})-\varkappa\Delta\boldsymbol{\omega}({x}))\cdot\widetilde{\boldsymbol{\eta}}({x},\widetilde{t})\Big)\, d{x} d\widetilde{t} 
\\[0.8ex]
&\quad-\int_\Omega\overline{\mathbf{v}}_n({x},1)
\cdot\widetilde{\boldsymbol{\eta}}({x},0)\,d{x}
-\varkappa\int_\Omega\nabla\overline{\mathbf{v}}_n({x},1)
:\nabla\widetilde{\boldsymbol{\eta}}({x},0)\,d{x}
=0.
\end{align*}
Using the energy estimate \eqref{eq:est-global-11} along with the Banach--Alaoglu theorem, we can extract subsequences of $\{\widetilde{\mathbf{v}}_n\}_{n \geq n_0}$, still labeled by $n$, such that the following convergences hold:
\begin{equation}\label{Weak conv}
\left\{
\begin{aligned}
|\mathbf{D}(\widetilde{\mathbf{v}}_n)|^{p-2}\mathbf{D}(\widetilde{\mathbf{v}}_n) \xrightharpoonup[n\to\infty]{} \mathbf{A} & \ \text{ weakly in } (L^{\frac{p}{p-1}}(Q_T))^{d^2},\\
\widetilde{\mathbf{v}}_n \xrightharpoonup[n\to\infty]{*} \mathbf{v} & \ \text{ weakly$^*$ in } L^\infty(0,T;\mathbb{V}).
\end{aligned}\right.
\end{equation}
Using the strong convergence of $\widetilde{\mathbf{v}}_n$ in $L^2(0,T;\mathbb{H})$, we can also choose a further subsequence, still labeled by $n$, such that
\begin{align*}
\widetilde{\mathbf{v}}_n(t) \xrightarrow[n\to\infty]{} \mathbf{v}(t) \ \ \text{in} \ \mathbb{H} \ \quad \text{ for a.e.} \ t \in (0,T).
\end{align*}
We obtain the following integral identity in the limit:
\begin{align*}
&\int_{Q_T} \Big(-\mathbf{v}({x},\widetilde{t}) \cdot \partial_{\widetilde{t}}\, \widetilde{\boldsymbol{\eta}}({x},\widetilde{t})
-\varkappa\nabla\mathbf{v}({x},\widetilde{t}):\nabla \partial_{\widetilde{t}}\, \widetilde{\boldsymbol{\eta}}({x},\widetilde{t})
- (\mathbf{v}({x},\widetilde{t}) \otimes\mathbf{v}({x},\widetilde{t})):\nabla \widetilde{\boldsymbol{\eta}}({x},\widetilde{t}) 
\\[0.8ex]
&\quad
+\mu\mathbf{A}({x},\widetilde{t}):\nabla \widetilde{\boldsymbol{\eta}}({x},\widetilde{t})
-\varphi(\widetilde{t})(\boldsymbol{\omega}({x})-\varkappa\Delta\boldsymbol{\omega}({x}))\cdot\widetilde{\boldsymbol{\eta}}({x},\widetilde{t})\Big)\, d{x} d\widetilde{t} 
\\[0.8ex]
&\qquad-\int_\Omega\mathbf{v}({x},1)
\cdot\widetilde{\boldsymbol{\eta}}({x},0)\,d{x}
-\varkappa\int_\Omega\nabla\mathbf{v}({x},1)
:\nabla\widetilde{\boldsymbol{\eta}}({x},0)\,d{x}
=0,
\end{align*}
with
\begin{align*}
\varphi(\widetilde{t})&=j'(\widetilde{t})
+\mu\int_\Omega \mathbf{A}(x,\widetilde{t}):\nabla\boldsymbol{\omega}({x})\,d{x}
-\int_\Omega
(\mathbf{v}({x},\widetilde{t})\otimes\mathbf{v}({x},\widetilde{t}))
:\nabla\boldsymbol{\omega}({x})\,d{x}.
\end{align*}
It is clear that the limit solution $(\mathbf{v},\varphi)$, where $\mathbf{v}=\lim\limits_{n\to\infty}\widetilde{\mathbf{v}}_n$ and $\varphi=\lim\limits_{n\to\infty}\widetilde{\varphi}_n$, satisfies the problem \eqref{eq3.6}. With the help of the monotonicity method, we show that $\mathbf{A}=
|\mathbf{D}(\mathbf{v})|^{p-2}\mathbf{D}(\mathbf{v}).$

\subsection{\bf Proof of item 3 in Theorem \ref{theorem3.1}}
\label{sec3.3}
\vspace{2mm}
In item 2, we have shown that the limit solution $\mathbf{v}(x,t)=\lim\limits_{n\to\infty}\widetilde{\mathbf{v}}_n(x,t)$ satisfies the new initial data
\[
\mathbf{v}({x},0)=\lim\limits_{n\to\infty}
\widetilde{\mathbf{v}}_n({x},0)=
\overline{\mathbf{v}}({x},1)
\]
in a strong sense. We have proven that the family $\{\partial_t\mathbf{v}_{n}\}_{n\geq n_0}$ is uniformly bounded in $L^1(0,T;\mathbb{V})$. This fact, along with the uniform boundedness of $\{\mathbf{v}_{n}\}_{n\geq n_0}$ in $L^\infty(0,T;\mathbb{V})$, is sufficient to apply the Aubin-Lions lemma in $L^p$ version, see Lemma \ref{ALSlemma}, and to prove the relative compactness of $\{\mathbf{v}_n\}_{n\geq n_0}$ in $L^2(0,T;\mathbb{H})$. There exists a subsequence that strongly converges in $L^2(0,T;\mathbb{H})$ to $\mathbf{v}$, which is the limit of $\{\widetilde{\mathbf{v}}_n\}_{n\geq n_0}$
\[
\int_0^T\|\mathbf{v}_n(\cdot,\widetilde{t})-
\widetilde{\mathbf{v}}_n(\cdot,\widetilde{t})\|_{\mathbb{H}}^2\,d\widetilde{t}\longrightarrow 0\text{ as }n\longrightarrow \infty,
\]
due to Corollary \ref{cor3.11}. Moreover, $\mathbf{1}_{(1/n,T)}(t)\varphi_{n}$ converges weakly in $L^2(0,T)$ to $\varphi$.

\section*{Conclusion}
 The existence result,  Theorem \ref{Theorem.2.1}, for a fixed $n\in\mathbb{N}$, is based on several well-known facts concerning inverse problems for the Kelvin--Voigt equations with an integral over-determination condition. In fact, we deal with the family of inverse problems, where the right-hand side $j_n$ approximates a discontinuous function as $n\to\infty$. In order to pass to the limit, we have obtained several uniform estimates.
The main novelty, Theorem \ref{theorem3.1}, is related to passing to the limit as $n\to\infty$. We apply the scaling $\overline{\mathbf{v}}_n(x,\overline{t})=\mathbf{v}_n({x},\overline{t}/n)$ to obtain the auxiliary inverse problem on the infinitesimal initial layer. 
In the limit, we obtain new initial data $\mathbf{v}({x},0)=\overline{\mathbf{v}}(x,1)$  defined by \eqref{eq3.5}, for the limit problem \eqref{eq3.6}.

\subsection*{\bf Acknowledgments}
The authors express their gratitude to K. Khompysh for useful advice and discussions.
S.N. Antontsev was supported by the Russian Science Foundation project no. 25-11-00056.  I. Kuznetsov was supported by the Ministry of Science and Higher Education of Kazakhstan under project no. AP26199323.

\subsection*{\bf Data Availability Statement} 
The authors state that no data were generated or analyzed in this paper, so data availability is not applicable.

\subsection*{\bf Declarations}

\subsection*{\bf Conflict of interest} The authors state that there is no
conflict of interest.

\bibliographystyle{elsarticle-num} 
\bibliography{bibl}
\end{document}